\documentclass[a4paper, 11pt, english]{article}
\usepackage[utf8]{inputenc}
\usepackage[T1]{fontenc}
\usepackage{graphicx}
\usepackage{stmaryrd}
\usepackage[a4paper]{geometry}
\usepackage{amsmath,amsfonts,amssymb,amsthm,epsfig,epstopdf,url,array}
\usepackage{rotating}
\usepackage[colorlinks=true,citecolor=red,linkcolor=blue,pdfpagetransition=Blinds]{hyperref}
\usepackage{cleveref}
\usepackage{nameref}
\usepackage{enumitem}
\usepackage{comment}
\Crefname{paragraph}{Section}{Sections}
\usepackage{fancyhdr}
\usepackage{booktabs}
\usepackage{multirow}
\usepackage{siunitx}

\usepackage{fullpage}

\usepackage{tikz}
\usepackage{pgfplots}
\usepackage{subfig}
\usetikzlibrary{matrix,external}

\newcommand{\ensemblenombre}[1]{\mathbb{#1}}

\newcommand{\R}{} 
\renewcommand{\R}{\ensemblenombre{R}}

\newcommand\D{\displaystyle}

\newcommand{\norme}[1]{\left\lVert#1\right\rVert}

\makeatletter
\newcommand{\compemb}{\mathrel{\mathpalette\comp@emb\relax}}
\newcommand{\comp@emb}[2]{%
  \vcenter{%
    \offinterlineskip\m@th
    \ialign{$#1##$\cr\hookrightarrow\cr\noalign{\vskip0.5pt}\hookrightarrow\cr}%
  }%
}
\makeatother

\theoremstyle{plain} 
\newtheorem{prop}{Proposition}[section] 
\newtheorem{thm}[prop]{Theorem}

\newtheorem{lem}[prop]{Lemma}

\theoremstyle{definition}

\newtheorem{rmk}[prop]{Remark}

\newtheorem{ass}[prop]{Assumption}

\def\dx{\,\textnormal{d}x}
\def\dt{\textnormal{d}t}
\def\d{\,\textnormal{d}}

\def\XXint#1#2#3{{\setbox0=\hbox{$#1{#2#3}{\int}$ }
\vcenter{\hbox{$#2#3$ }}\kern-.6\wd0}}

\makeatletter
\let\original@addcontentsline\addcontentsline
\newcommand{\dummy@addcontentsline}[3]{}
\newcommand{\DeactivateToc}{\let\addcontentsline\dummy@addcontentsline}
\newcommand{\ActivateToc}{\let\addcontentsline\original@addcontentsline}
\makeatother

\begin{document}

\title{Local-controllability of the one-dimensional nonlocal Gray-Scott model with moving controls}
\author{V\'ictor Hern\'andez-Santamar\'ia \and Kévin Le Balc'h}

\maketitle

\begin{abstract}
In this paper, we prove the local-controllability to positive constant trajectories of a nonlinear system of two coupled ODE equations, posed in the one-dimensional spatial setting, with nonlocal spatial nonlinearites, and using only one localized control with a moving support. The model we deal with is derived from the well-known nonlinear reaction-diffusion Gray-Scott model when the diffusion coefficient of the first chemical species $d_u$ tends to $0$ and the diffusion coefficient of the second chemical species ${d_v}$ tends to $+ \infty$. The strategy of the proof consists in two main steps. First, we establish the local-controllability of the reaction-diffusion ODE-PDE derived from the Gray-Scott model taking $d_u=0$, and uniformly with respect to the diffusion parameter ${d_v} \in (1, +\infty)$. In order to do this, we prove the (uniform) null-controllability of the linearized system thanks to an observability estimate obtained through adapted Carleman estimates for ODE-PDE. To pass to the nonlinear system, we use a precise inverse mapping argument and, secondly, we apply the shadow limit ${d_v} \rightarrow + \infty$ to reduce to the initial system.
\end{abstract}
\small
\tableofcontents
\normalsize

\section{Introduction}\label{sec:intro}

\subsection{The Gray-Scott model}

The irreversible Gray-Scott model governs the chemical reaction 
\begin{align}
\label{eq:chemicalcubic}
\mathcal{U} + 2 \mathcal{V}& \rightarrow 3 \mathcal{V}, \\
\mathcal{V} &\rightarrow \mathcal{P},\label{eq:chemicalinert}
\end{align}
in a gel reactor, where $\mathcal{V}$ catalyses its own reaction with $\mathcal{U}$ and $\mathcal{P}$ is an inert product. The gel reactor is coupled to a reservoir in which the concentrations of $\mathcal{U}$ and $\mathcal{V}$ are maintained constant.

Let $T>0$ be a (positive) time. For $t$ in the time interval $[0,T]$ and $x$ in the spatial interval $[0,1]$, i.e. the gel reactor, we denote $u=u(t,x)$ and $v(t,x)$ the concentrations of the two chemical species $\mathcal{U}$ and $\mathcal{V}$, $d_u, {d_v} >0$ their (constant) diffusion coefficients. Then, the pair of coupled reaction-diffusion equations governing theses reactions are
\begin{equation}\label{eq:classical_GSS}
\begin{cases}
\partial_t u - d_u \partial_{xx}u = -uv^2 + F(1-u) & \text{in } (0,T)\times(0,1), \\
\partial_t v -{d_v}  \partial_{xx}v = uv^2 - (F+k) v & \text{in } (0,T)\times(0,1),  \\
\partial_x u=\partial_x v=0 &\text{on } (0,T)\times\{0,1\}, \\
u(0,x)=u_0(x), \quad v(0,x)=v_0(x) &\text{in }(0,1).
\end{cases}
\end{equation}
In \eqref{eq:classical_GSS}, the cubic terms $uv^2$ and $-uv^2$ correspond to the chemical reaction \eqref{eq:chemicalcubic}, seeing $u$ as a reactant and $v$ as a product, by the law of mass action. The linear term $k v$ comes from the chemical reaction \eqref{eq:chemicalinert} at rate $k>0$ and the positive constant $F>0$ denotes the rate at which $\mathcal{U}$ is fed from the reservoir into the reactor (and this same feed process takes $\mathcal{U}$ and $\mathcal{V}$ out in a concentration-dependent way). We impose homogeneous Neumann boundary conditions on $u$ and $v$ to guarantee that the environment is closed. See \cite{DKZ97} and the references therein for more details on the Gray-Scott model.\\
\indent Note that the global existence of classical solutions for \eqref{eq:classical_GSS} follows from classical bootstrap argument because the spatial dimension is one. To obtain the global existence of classical solutions for \eqref{eq:classical_GSS} in all spatial dimension, one can use for instance \cite[Theorem 3.1]{Pie10}.\\
\indent The steady states of the local dynamics (i.e. setting $d_u={d_v}=0$) give us the uniform steady-state solutions to \eqref{eq:classical_GSS}. The point $(1,0)$ is an uniform steady state. Moreover, if $F\geq 4(F+k)^2$ or $F=\frac{1}{2}[(1/4)-2k\pm\sqrt{1/16-k}]$, there exits two steady states for \eqref{eq:nonlinear_shadow}, $(u_{+},v_{-})$ and $(u_{-},v_{+})$, characterized as follows
\begin{equation}\label{eq:stationary}
u_{\pm}=\frac{1}{2}(1\pm\sqrt{1-4\gamma^2F}), \ v_{\mp}=\frac{1}{2\gamma}(1\mp\sqrt{1-4\gamma^2F}), \quad \text{with}\ \gamma=\frac{F+k}{F}.
\end{equation}
The proof of this fact can be found, for instance, in \cite[Section 3]{MGR04}.\\
\indent For some chemical species $\mathcal{U}$ and $\mathcal{V}$, one can have $d_u << 1 << d_v$. If we send $(d_u, {d_v}) \rightarrow (0,+\infty)$, we obtain formally the following nonlocal Gray-Scott model
\begin{equation}\label{eq:nonlinear_shadow_SC}
\begin{cases}
\partial_t u=-uw^2 + F(1-u) & \text{in } (0,T)\times(0,1), \\
\D w^\prime=\left(\int_{0}^{1} u(t,x)\dx\right) w^2 - (F+k) w &\text{in }(0,T), \\
u(0,x)=u_0(x)\ \text{in }(0,1), \quad w(0)=w_0.
\end{cases}
\end{equation}

The main goal of this paper is to study some controllability properties for \eqref{eq:nonlinear_shadow_SC}. More precisely, we introduce the following distributed controlled system
\begin{equation}\label{eq:nonlinear_shadow}
\begin{cases}
\partial_t u=-uw^2 + F(1-u)+ h \mathbf{1}_{\omega(t)} & \text{in } (0,T)\times(0,1), \\
\D w^\prime=\left(\int_{0}^{1} u(t,x)\dx\right) w^2 - (F+k) w&\text{in }(0,T), \\
u(0,x)=u_0(x)\ \text{in }(0,1), \quad w(0)=w_0.
\end{cases}
\end{equation}
In \eqref{eq:nonlinear_shadow}, at $(t,x) \in (0,T) \times (0,1)$, $(u(t,x), w(t))$ is the \textit{state} while $h(t,x)$ is the \textit{control input}, whose support is localized in a moving subset $\omega(t)$ of $(0,1)$.\\
\indent Typically, we shall consider controls sets $\omega(t)$ determined by the evolution of a given reference $\omega$ of $(0,1)$ through a smooth flow $X(t,x,0)$. This type of support for the control variable $h$ is justified by the fact that there is no diffusion in the variable $u(t,x)$ in \eqref{eq:nonlinear_shadow_SC} so for obtaining controllability, at least in the variable $u$, one needs to consider moving control support as in the articles \cite{MRR13,CSRZ14,CSZZ17,KSD18} for instance.

\subsection{Main results}

The goal of this part is to present the main results of the paper. Among others, we prove the local-controllability to constant positive trajectories for \eqref{eq:nonlinear_shadow_SC}.\\
\indent First, we introduce some assumption for the moving support $\omega(t)$ of the control variable $h$ in \eqref{eq:nonlinear_shadow_SC}.
\begin{ass}\label{ass:moving_control}
There exist a subset $\omega_0 \subset \omega$, two times $t_1$, $t_2$ with $0<t_1<t_2<T$ such that
\begin{enumerate}
\item[a)] $\omega_0(t)\neq (0,1)$ for all $t\in(0,T)$,
\item[b)] $\bigcup_{t\in(0,T)} \omega_0(t)=(0,1)$,
\item[c)] $(0,1)\setminus \omega_0(t)$ is nonempty and connected in $(0,1)$ for any $t\in(0,t_1]\cup[t_2,T)$,
\item[d)] $(0,1)\setminus \omega_0(t)$ has two nonempty connected components in $(0,1)$ for any $t\in(t_1,t_2)$.
\end{enumerate}
\end{ass}
Typically, for every $m \in (0,1) $, the set $\omega_m = \{(0,m) + t (T-m)/T :  t \in (0,T)\}$ satisfies \Cref{ass:moving_control}.\\
\indent The first main result of the paper is the following one.
\begin{thm}
\label{th:mainresult1}
We suppose that \Cref{ass:moving_control} holds. Let $(u_{\pm}, v_{\mp}) \in (0,+\infty)^2$ as in \eqref{eq:stationary}. Then there exists $\delta >0$ such that for every $(u_0, w_0) \in L^2(0,1) \times \R$, verifying $$\norme{(u_0-u_{\pm}, w_0 - v_{\mp})}_{L^2(0,1) \times \R} \leq \delta,$$ one can find a control $h \in L^2((0,T)\times(0,1))$ such that the solution $(u, w) \in H^1(0,T;L^2(\Omega)) \times L^{\infty}(0,T)$ of \eqref{eq:nonlinear_shadow} satisfies
\begin{equation}
\label{eq:finaluw}
(u, w)(T) = (u_{\pm}, v_{\mp}).
\end{equation}
\end{thm}
\begin{rmk}
Let us make some comments on \Cref{th:mainresult1}.
\begin{itemize}
\item For a given control $h \in L^2((0,T)\times(0,1))$, note that the solutions $(u, w)$ to \eqref{eq:nonlinear_shadow} belonging to $ H^1(0,T;L^2(\Omega)) \times L^{\infty}(0,T)$ are necessarily unique by using classical Gronwall's argument. The existence of such a solution, associated to some specific control $h$, actually comes from the proof of \Cref{th:mainresult1}.
\item From a modelling point of view, \Cref{th:mainresult1} states that for diffusion coefficients $0 < d_u << 1 << {d_v}$ associated to the chemical species $\mathcal{U}$ and $\mathcal{V}$, by starting from chemical concentrations closed to a chemical equilibrium, there exists a strategy of control, i.e. by adding or withdrawing some chemical product at some moving place of the gel reactor, such that the chemical components $\mathcal{U}$ and $\mathcal{V}$ exactly reach the chemical equilibrium. This is particularly relevant when $(u_{\pm}, v_{\mp})$ is an unstable equilibrium of \eqref{eq:classical_GSS}, see \cite[Section 3]{MGR04}.
\item  \Cref{ass:moving_control} is a natural hypothesis for dealing with the controllability of systems of the form \eqref{eq:nonlinear_shadow_SC}. Indeed, this ODE-ODE system has a finite speed of propagation so the time $T>0$ is taken sufficiently large such that the initial support of control $\omega(0)$ spreads the whole interval $(0,1)$. This is exactly \Cref{ass:moving_control}, b).
\item Let us remark that $(u_{*}, v_{*}) = (1,0)$ is also a constant stationary state of \eqref{eq:classical_GSS}. But \eqref{eq:nonlinear_shadow} is not locally controllable around $(1,0)$. This comes from the fact that all solution $(u,w) \in H^1(0,T;L^2(\Omega)) \times L^{\infty}(0,T)$ to \eqref{eq:nonlinear_shadow}, reaching $(1,0)$ in time $T$, satisfies necessarily $w \equiv 0$ in $(0,T)$. Indeed, setting $a(t) = \int_{0}^1 u(t,x) \dx \in L^{\infty}(0,T)$,  rewriting the second equation, we obtain
$$ w' = a(t) w^2(t) - (F+k) w(t),\ w(T) = 0.$$
So, by the Cauchy-Lipschitz theorem, we obtain that $w \equiv 0$ in $(0,T)$.
\item As far as we know, \Cref{th:mainresult1} is the first result in the literature which deals with the controllability of nonlinear system of coupled ODE equations, posed in the one-dimensional spatial setting, with nonlocal spatial nonlinearites. For results on the controllability of linear and nonlinear parabolic PDEs with spatially nonlocal terms, see \cite{FCLZ16}, \cite{LZ18}, \cite{BHZ19}, \cite{FCLNHNC19} and the recent article of the authors \cite{HSLB20}.
\end{itemize}
\end{rmk}
\indent Actually, a byproduct of the proof of \Cref{th:mainresult1} is a local-controllability result for the following reaction-diffusion ODE-PDE model
\begin{equation}\label{eq:semilinear_diffusion}
\begin{cases}
\partial_t u=-uv^2 + F(1-u)+ h \mathbf{1}_{\omega(t)} & \text{in } (0,T)\times(0,1), \\
\D \partial_t v- {d_v}\partial_{xx}v  =u v^2 - (F+k) v &\text{in }(0,T)\times(0,1), \\
\partial_x v=0 &\text{on } (0,T)\times \{0,1\}, \\
(u,v)(0,\cdot)=(u_0,v_0) &\text{in }(0,1),
\end{cases}
\end{equation}
for ${d_v} \in (1,+ \infty)$.
\begin{thm}
\label{th:mainresult2}
We suppose that \Cref{ass:moving_control} holds. Let $(u_{\pm}, v_{\mp}) \in (0,+\infty)^2$ as in \eqref{eq:stationary}. Then there exist $\delta >0$, $C>0$ such that for every ${d_v} \in (1, + \infty)$, $(u_0, v_0) \in L^2(\Omega) \times H^1(\Omega)$, verifying $$\norme{(u_0-u_{\pm}, v_0 - v_{\mp})}_{L^2(\Omega) \times H^1(\Omega)} \leq \delta,$$ one can find a control $h \in L^2((0,T)\times(0,1))$ such that the solution $(u,v)$ of \eqref{eq:semilinear_diffusion} satisfies
\begin{equation}
\label{eq:boundedsigmauvh}
\norme{u}_{H^1(0,T;L^2(0,1))} + \norme{v}_{ L^{\infty}(0,T;H^1(0,1)) \cap H^1(0,T;L^2(0,1))} + \norme{h} _{L^2(0,T;L^2(0,1))} \leq C,
\end{equation}
and
\begin{equation}
\label{eq:finaluv}
(u, v)(T) = (u_{\pm}, v_{\mp}).
\end{equation}
\end{thm}

\begin{rmk}
Let us make some comments on \Cref{th:mainresult2}.
\begin{itemize}
\item For a given control $h \in L^2((0,T)\times(0,1))$, note that the solutions $(u, v)$ to \eqref{eq:semilinear_diffusion} belonging to $ H^1(0,T;L^2(\Omega)) \times L^{\infty}((0,T)\times (0,1))$ are necessarily unique by using classical Gronwall's argument. The existence of such a solution, associated to some specific control $h$, actually comes from a precise inverse mapping argument performed in the proof of \Cref{th:mainresult2}.
\item \Cref{th:mainresult2} is a uniform local-controllability result with respect to the parameter ${d_v} \rightarrow + \infty$.
\item By adding a diffusion term $-d_u\partial_{xx}u$, $d_u >0$, in the first equation of \eqref{eq:semilinear_diffusion}, fixing $d_v >0$ and setting $\omega(t) = \omega$, any arbitrary nonempty open set contained in $(0,1)$, we can easily adapt the proof of \Cref{th:mainresult2} to obtain a local-controllability result to positive constant trajectories for the classical reaction-diffusion Gray-Scott model (cf. \cite{AKBD06}).
\item As for \eqref{eq:nonlinear_shadow}, system \eqref{eq:semilinear_diffusion} is not locally controllable around $(1,0)$, at least for smooth solutions. This comes from the fact that all solutions $(u,v) \in L^{\infty}((0,T)\times(0,1)) \times L^{\infty}(0,T)$ to \eqref{eq:nonlinear_shadow}, reaching $(1,0)$ in time $T$, satisfies necessarily $v \equiv 0$ in $(0,T)\times(0,1)$. Indeed, setting $a(t,x) = uv - (F+k) \in L^{\infty}((0,T)\times(0,1))$,  rewriting the second equation, we obtain
$$ \partial_t v - d_{v} \partial_{xx} v  = a(t,x) v,\ v(T) = 0.$$
So, by backward uniqueness for parabolic equation, we obtain that $v \equiv 0$ in $(0,T)\times(0,1)$.
\end{itemize}
\end{rmk}

\subsection{Strategy of the proof}

In order to prove \Cref{th:mainresult1}, a natural strategy would be to linearize \eqref{eq:nonlinear_shadow} around $(u_{\pm}, v_{\mp})$ to obtain
\begin{equation}\label{eq:linear_shadow}
\begin{cases}
\partial_t u=(-v_{\mp}^2-F) u - 2u_{\pm} v_{\mp} w+ h \mathbf{1}_{\omega(t)} & \text{in } (0,T)\times(0,1), \\
\D w^\prime= v_{\mp}^2 \left(\int_{0}^{1} u(t,x)\dx\right)  + ( 2 u_{\pm} v_{\mp} -( F+k)) w&\text{in }(0,T), \\
u(0,x)=u_0(x)\ \text{in }(0,1), \quad w(0)=w_0.
\end{cases}
\end{equation}
Heuristically, \eqref{eq:linear_shadow} seems to be null-controllable because the control $h$ controls the first component $u$ and the nonlocal coupling term $v_{\mp}^2 \left(\int_{0}^{1} u(t,x)\dx\right)$ indirectly controls the second component $w$. But, as far as we know, they do not exist classical tools in the literature to deal with the null-controllability of such a nonlocal system. That is why we follow a different approach to prove \Cref{th:mainresult1}.

\smallskip
The method we employ for proving \Cref{th:mainresult1} is based on two key points.

First, we prove \Cref{th:mainresult2}. In order to do this, we mainly follow \cite{KSD18} which establish the local-controllability to trajectories for a nonlinear system of ODE-PDE in $1$-D. However, two main differences appear comparing \Cref{th:mainresult2} and \cite[Theorem 1.1]{KSD18}. The first one is the uniformity of the local-controllability of \eqref{eq:semilinear_diffusion} with respect to the parameter ${d_v} \in (1, +\infty)$. The second one is the localization of the control in the ODE equation of \eqref{eq:semilinear_diffusion}, instead of the parabolic equation for \cite[System (7)]{KSD18}.\\
\indent  We give the main steps of the proof of \Cref{th:mainresult2}. 
\begin{itemize}
\item We first linearize \eqref{eq:semilinear_diffusion} around $(u_{\pm}, v_{\mp})$, this leads us to study the uniform null-controllability of the linearized system satisfied by the variable $(U,V) = (u - u_{\pm}, v - v_{\mp})$.  This is done in \Cref{sec:nulllin}.
\item We prove a uniform observability estimate for the adjoint system of the linearized equations obtained in the previous step. This is done thanks to a uniform Carleman estimate, which is inspired in the arguments of \cite{CSRZ14, KSD18}. We highlight the fact that the restriction to the one spatial dimensional case appears in this part because Carleman estimates with similar weights for ODE-PDE have only been proved in $1$-D when considering homogeneous Neumann boundary conditions. This is done in \Cref{sec:obs}
\item We deduce from the observability estimate and classical duality arguments, the null-controllability of the linearized system with a source term, exponentially decreasing at $t = T$, see for instance \cite[Theorem 2.44]{Cor07a} when the source is equal to zero. We also prove some extra regularity results on the controlled trajectory, this part is actually crucial to pass to the nonlinear result. This type of argument is inspired from \cite[Chapter I, Section 4]{fursi} and \cite{CSG15}. This is done in \Cref{sec:nullsource}.
\item We use a precise inverse mapping argument to deduce from the (global) null-controllability of the linearized system a local null-controllability result for the nonlinear system satisfied by $(U,V) = (u - u_{\pm}, v - v_{\mp})$. Note that the regularity of the nonlinear mapping is obtained thanks to the extra regularity of the linear controlled trajectory proved in the previous step. This is done in \Cref{sec:inversemappingarg}.
\end{itemize}

\indent Secondly, we prove \Cref{th:mainresult1} by using \Cref{th:mainresult2} and the shadow limit method. Roughly, we obtain that the solution $(u_{{d_v}}, v_{{d_v}}, h_{{d_v}})$ of \eqref{eq:semilinear_diffusion} converges in some sense to $(u,w,h)$ the solution of \eqref{eq:nonlinear_shadow} as ${d_v} \rightarrow + \infty$. This method relies on an adaptation of the arguments presented in \cite[Appendix A]{MCHKS18}. For the use of such a method in the context of control theory, see \cite{HSZ20} and \cite{HSLB20}.\\

\section{Null-controllability of the linearized system}
\label{sec:nulllin1}
\subsection{Change of variable and linearized system}
\label{sec:nulllin}
By setting $(U,V) = (u - u_{\pm}, v - v_{\mp})$, where $(u,v)$ is the solution to \eqref{eq:semilinear_diffusion}, we obtain that $(U,V)$ satisfies
\begin{equation}\label{eq:nonlinear_diffusionUV}
\begin{cases}
\partial_t U= a_{11} U + a_{12} V + N_1(U,V) +   h \mathbf{1}_{\omega(t)} & \text{in } (0,T)\times(0,1), \\
\partial_t V- {d_v}\partial_{xx}V  = a_{21} U + a_{22} V + N_2(U,V)  &\text{in }(0,T)\times(0,1), \\
\partial_x V=0 &\text{on } (0,T)\times \{0,1\}, \\
(U,V)(0,\cdot)=(U_0,V_0) &\text{in }(0,1),
\end{cases}
\end{equation}
with 
\begin{equation}
\label{eq:defaij}
a_{11} = - v_{\mp}^2 - F,\ a_{12} = - 2 u_{\pm} v_{\mp},\ a_{21} = v_{\mp}^2,\ a_{22} = 2 u_{\pm} v_{\mp} - (F+k),
\end{equation}
\begin{equation}
N(U,V) := \begin{pmatrix}
N_1(U,V) \\
N_2(U,V) 
\end{pmatrix} := \begin{pmatrix}
-(UV^2 + 2 v_{\mp} UV + u_{\pm} V^2) \\
 UV^2 + 2 v_{\mp} UV + u_{\pm} V^2
\end{pmatrix}.\label{eq:NonlinearN}
\end{equation}

The goal of \Cref{sec:nulllin1} is to prove the null-controllability of the linearized system
\begin{equation}\label{eq:linearized_diffusion}
\begin{cases}
\partial_t U=a_{11} U + a_{12} V + F_1 +   h \mathbf{1}_{\omega(t)} & \text{in } (0,T)\times(0,1), \\
\partial_t V- {d_v}\partial_{xx}V  = a_{21} U + a_{22} V + F_2 &\text{in }(0,T)\times(0,1), \\
\partial_x V=0 &\text{on } (0,T)\times \{0,1\}, \\
(U,V)(0,\cdot)=(U_0,V_0) &\text{in }(0,1),
\end{cases}
\end{equation}
where $F_1$, $F_2$ are source terms belonging to an appropriate Banach space $X$ and exponentially decreasing at $t=T$, see \Cref{prop:NullControllabilityLinear} below. This would be indeed possible thanks to the fact that 
\begin{equation}
\label{eq:couplage}
a_{21} \neq 0,
\end{equation}
using \eqref{eq:defaij} because $v_{\mp} \neq 0$. Heuristically, the control $h$ directly controls the component $U$ thanks to the first equation of \eqref{eq:linearized_diffusion} and $U$ indirectly controls the component $V$ thanks to the coupling term $a_{21} V$, appearing in the second equation of \eqref{eq:linearized_diffusion}.

Our objective then will be to prove that we can find $h$, bounded independently of ${d_v} \in (1, +\infty)$ such that the solution $(U,V)$ of \eqref{eq:linearized_diffusion} satisfies $(U,V)(T) = 0$. Moreover, we want that the nonlinear quantity $N(U,V)$ belongs to $X$
to employ at the end of the day an inverse mapping argument to obtain the controllability of \eqref{eq:semilinear_diffusion} around $(u_{\pm}, v_{\mp})$, see \Cref{sec:inversemappingarg} below.\\
\indent In the sequel, we will use the following notations
$$\Omega := (0,1),\  Q_T := (0,T)\times (0,1)\ \text{and}\ \Sigma_T = (0,T)\times\{0,1\}.$$

The following  standard proposition, stated without proof, guarantees the well-posedness of \eqref{eq:linearized_diffusion}. It can be established for instance using Galerkin approximations.
\begin{prop}
For every $(F_1,F_2) \in L^2(Q_T)^2$, $(U_0,V_0) \in L^2(\Omega)^2$, the system \eqref{eq:linearized_diffusion} admits a unique weak solution $(U,V) \in [H^1(0,T;L^2(\Omega))] \times [L^2(0,T;H^1(\Omega)) \cap H^1(0,T;H^1(\Omega)')]$.
\end{prop}
\subsection{Observability of the adjoint system}
\label{sec:obs}
In order to prove the null-controllability of the linearized system \eqref{eq:linearized_diffusion}, we will first prove an observability inequality for the following adjoint system
\begin{equation}\label{eq:linear_adj}
\begin{cases}
-\partial_t\phi =a_{11} \phi + a_{21}\psi + g_1 & \text{in } Q_T, \\
-\partial_t\psi={d_v}\partial_{xx}\psi +a_{12}\phi+a_{22} \psi +g_2 &\text{in }Q_T, \\
\partial_x \psi=0 &\text{on } \Sigma_T, \\
(\phi,\psi)(T,\cdot)=(\phi_T,\psi_T) &\text{in }(0,1).
\end{cases}
\end{equation}
where $g_1, g_2 \in L^2(0,T;L^2(0,1))$ are given source terms. More precisely, the goal is to prove that there exists a positive constant $C >0$ such that for every $(\phi_T, \psi_T) \in L^2(\Omega)^2$, 
\begin{equation}
\label{eq:ObsDesiredL2}
\norme{(\phi,\psi)(0, \cdot)}_{L^2(\Omega)^2} \leq C \int_{0}^T\int_{\omega(t)} |\phi(t,x)|^2 \dt \dx.
\end{equation} 
Note that the null-controllability of \eqref{eq:linearized_diffusion} when $F_1 = F_2 = 0$ is a direct consequence of \eqref{eq:ObsDesiredL2}, thanks to a duality argument, called Hilbert Uniqueness Method, see \cite[Theorem 2.44]{Cor07a}.\\
\indent For proving such an observability inequality \eqref{eq:ObsDesiredL2}, we will use Carleman estimates.

\subsubsection{Preliminaries on Carleman inequalities with moving controls}

The ideas presented below were mainly developed in \cite{CSZZ17,CSRZ14,KSD18}. In particular, the results presented in \cite{KSD18} are at the heart of our methodology. 

We begin by recalling that $\omega$ satisfies \Cref{ass:moving_control}.  Let us consider $\omega_{i}:[0,T]\to 2^{(0,1)}$, $i=1,2$, subsets of $\omega$, such that
\begin{equation}
\label{eq:subsetomega}
\overline{\omega}_0\subset\mathring{\overline{\omega}}_1, \quad  \overline{\omega}_1\subset \mathring{\overline{\omega}}_2, \quad\text{and } \overline{\omega}_2\subset \mathring{\overline{\omega}},
\end{equation}
where $\mathring{\overline{\omega}}_i$, $i=1,2$ and $\mathring{\overline{\omega}}$ stand for the relative interiors with respect to $[0,1]\times[0,T]$ of $\overline{\omega_i}$ and $\overline{\omega}$, respectively. 

As usual in the framework of Carleman estimates, the first step is to construct a suitable weight function. The function introduced below allows to obtain a Carleman inequality for parabolic equations coupled to ordinary differential equations, which is exactly the structure of the adjoint system \eqref{eq:linear_adj}.

\begin{lem}\label{lem:weight_function}
 There exist a positive number $\tau\in(0,\min\{1,T/2\})$, a positive constant $C_0>0$, and a function $\eta\in C^\infty([0,1]\times[0,T])$ such that
\begin{align}\label{eq:deriv_eta}
\eta_x(x,t)\neq 0 &\qquad \forall x\in\overline{(0,1)\setminus \omega_0(t)}, \ \forall t\in[0,T],  \\
\eta_t(x,t)\neq 0 &\qquad \forall x\in\overline{(0,1)\setminus \omega_0(t)}, \ \forall t\in[0,T],\label{eq:dtetaneq}   \\
\eta_t(x,t)> 0 &\qquad \forall x\in\overline{(0,1)\setminus \omega_0(t)}, \ \forall t\in[0,\tau],\label{eq:dtetapos}   \\
\eta_t(x,t)< 0 &\qquad \forall x\in\overline{(0,1)\setminus \omega_0(t)}, \ \forall t\in[T-\tau,T], \label{eq:dtetaneg} \\ \label{eq:deriv_w_0}
\eta_x(0,t)\geq C_0 &\qquad \forall t\in[0,T], \\ \label{eq:deriv_w_1}
\eta_x(1,t)\leq -C_0 & \qquad \forall t\in [0,T], \\
\min_{(x,t)\in[0,1]\times[0,T]}\{\eta(x,t)\}&=\frac{3}{4}\|\eta\|_{L^\infty([0,1]\times(0,T))}
\end{align}
\end{lem}
The proof of Lemma \ref{lem:weight_function} can be obtained as in \cite[Lemma 4.3]{CSRZ14} with the observation of \cite[Lemma 1]{KSD18} stating that precise values for the derivative of the weight at the boundary \eqref{eq:deriv_w_0}--\eqref{eq:deriv_w_1} have a prescribed sign. 

Now, let us introduce a function $r\in C^\infty(0,T)$, symmetric with respect to $t=\frac{T}{2}$ (more precisely, $r(t)=r(T-t)$ for any $t\in(0,T)$) and such that
\begin{equation}\label{eq:def_r}
r(t)=\begin{cases}
\frac{1}{t} &\text{for } 0<t\leq \tau/2, \\
\textnormal{strictly decreasing} &\text{for } \frac{\tau}{2}<t<\tau, \\
1 &\text{for } \tau\leq t\leq T/2.
\end{cases}
\end{equation}

For any parameter $\lambda>0$, let us define the weights
\begin{equation}
\label{eq:defweights}
\alpha(x,t):=r(t)(e^{2\lambda\|\eta\|_\infty}-e^{\lambda\eta(x,t)}) \quad\text{and}\quad \xi(x,t):=r(t)e^{\lambda\eta(x,t)}, \quad\forall(x,t)\in Q_T
\end{equation}

We have the following uniform Carleman estimate for the heat equation with homogeneous Neumann boundary conditions.
\begin{lem}
\label{lem:CarlPDE}
For any $0<\epsilon\leq 1$, there exists positive constants $\lambda_1$, $s_1$ and $C$, depending on $\omega_1$, such that for any $\lambda\geq\lambda_1$, $s\geq s_1(\lambda)$, the solution $\psi$ to 
\begin{equation*}
\begin{cases}
\D -\partial_t\psi-\frac{1}{\epsilon}\partial_{xx}\psi = f &\textnormal{in } Q_T, \\
\psi_x=0 &\textnormal{on } \Sigma_T, \\
\psi(T,\cdot)=\psi_T &\textnormal{in } (0,1).
\end{cases}
\end{equation*}
with $\psi_T\in L^2(0,1)$ and $f\in L^2(0,T;L^2(0,1))$ verifies
\begin{align}\label{eq:car_PDE_unif}
I(\psi;\epsilon)\leq C\left(\epsilon^2 \iint_{Q_T}|f|^2e^{-2s\alpha}\dx\dt+s^3\lambda^4\iint_{\omega_1(t)\times(0,T)}\xi^3|\psi|^2\dx\dt\right)
\end{align}
where $I(\psi,\epsilon)$ stands for
\begin{align*}
I(\psi;\epsilon):=&\ s^{-1}\iint_{Q_T}\xi^{-1}(\epsilon^2 |\partial_t\psi|^2+|\partial_{xx}\psi|^2)e^{-2s\alpha}\dx\dt \\ &+s\lambda^2\iint_{Q_T}\xi|\partial_x\psi|^2e^{-2s\alpha}\dx\dt +s^3\lambda^4\iint_{Q_T}\xi^3|\psi|^2e^{-2s\alpha}\dx\dt \\
&+\left. s^3\lambda^3\int_{0}^{T}(\xi^2|\psi|^2e^{-2s\alpha})\dt\right|_{x=1}+\left. s^3\lambda^3\int_{0}^{T}(\xi^2|\psi|^2e^{-2s\alpha})\dt\right|_{x=0}.
\end{align*}
\end{lem}
The proof of this result follows the methodology of \cite[Appendix A]{KSD18} and pays special attention to the dependency of $\epsilon$ during the computations. We give a sketch of the proof in \Cref{sec:uniformcarlHeat}. Note that the important properties of the weights $\eta$ for obtaining the parabolic Carleman estimate \eqref{eq:car_PDE_unif} are \eqref{eq:deriv_eta}, \eqref{eq:deriv_w_0}, \eqref{eq:deriv_w_1}.\\

We have the following Carleman estimate for ODE, coming from \cite[Lemma 4.5]{CSRZ14}.
\begin{lem}
There exist some numbers $\lambda_1 \geq \lambda_0$, $s_1 \geq s_0$ and $C_1>0$ such that for all $\lambda \geq \lambda_1$, all $s \geq s_1$ and all $q \in H^1(0,T;L^2(0,1))$, the following holds
\begin{align}\notag
I(q):=s\lambda^2 \iint_{Q_T} & \xi |q|^2 e^{-2s \alpha} \dx \dt \\ \label{eq:car_ODE}
&\leq C_1 \left(\iint_{Q_T} |q_t|^2 e^{-2 s \alpha} \dx\dt + \lambda^2 \iint_{\omega_1(t)\times(0,T)} (s \xi)^2 |q|^2 e^{-2 s \alpha} \dx \dt \right).
\end{align}
\end{lem}
Note that the important properties of the weights $\eta$ for obtaining the ODE Carleman estimate \eqref{eq:car_ODE} are \eqref{eq:dtetaneq}, \eqref{eq:dtetapos}, \eqref{eq:dtetaneg}.

\subsubsection{An uniform observability inequality}
Let us introduce the following useful notations 
\begin{equation}
\begin{split}
\label{eq:defminmax}
\alpha^\star(t)&=\min_{x\in[0,1]}\alpha(x,t), \quad \widehat{\alpha}(t)=\max_{x\in[0,1]}\alpha(x,t), \\
\xi^\star(t)&=\max_{x\in[0,1]}\xi(x,t), \quad \widehat{\xi}(t)=\min_{x\in[0,1]}\xi(x,t).
\end{split}
\end{equation}

We have the following uniform Carleman estimate for the solution to \eqref{eq:linear_adj}.
\begin{prop}
\label{prop:UniformCarleman}
There exist positive constants $\lambda_2>0$, $s_2>0$ and $C>0$, such that for any ${d_v} \geq 1$, $\lambda\geq\lambda_2$, $s\geq s_2(\lambda)$ and any initial data $(\phi_T,\psi_T) \in L^2(\Omega)^2$, the solution to \eqref{eq:linear_adj} verifies
\begin{align} \notag 
&s \iint_{Q_T}  e^{-2s \alpha} \xi |\phi|^2  \dx \dt + s^3\iint_{Q_T}e^{-2s\alpha}\xi^3|\psi|^2\dx\dt \\ \notag
&\quad \leq C \Bigg( s^8\iint_{\omega_2(t)\times(0,T)}  e^{-4s\alpha^\star+2s\widehat{\alpha}}({\xi}^\star)^{8}|\phi|^2\dx\dt \\
&\qquad\qquad + s^3\iint_{Q_T}e^{-2s\alpha}\xi^3|g_1|^2\dx\dt  + \iint_{Q_T}e^{-2s\alpha} |g_2|^2\dx\dt  \Bigg).
\label{eq:car_Prop}
\end{align}
\end{prop}
\begin{proof}

We divide the proof in several steps. In the following, the positive constants $C>0$ vary from line to line and are independent of the parameters $d_v, \lambda, s$.

\subsubsection*{Step 1: First estimates}
We apply the Carleman estimate \eqref{eq:car_ODE} to the first equation of \eqref{eq:linear_adj} and obtain for $s$ and $\lambda$ large enough
\begin{align}
\label{eq:carODEProof1}
I(\phi) \leq C\left(\iint_{Q_T}e^{-2s\alpha}(|\phi|^2 + |\psi|^2+|{g_1}|^2\dx\dt+ \iint_{\omega_1(t)\times(0,T)}e^{-2s\alpha}(s\lambda\xi)^2\left(|\phi|^2 \right)\dx\dt\right).
\end{align}

We apply inequality \eqref{eq:car_PDE_unif} to the PDE in \eqref{eq:linear_adj} with $\epsilon=1/{d_v}$, to obtain
\begin{align} \label{eq:car_init_PDE}
I(\psi;{d_v}^{-1})\leq C\left({d_v}^{-2} \iint_{Q_T} \left(|\phi|^2+|\psi|^2+|g_2|^2\right)e^{-2s\alpha}\dx\dt+s^{3}\lambda^{4}\iint_{\omega_1(t)\times(0,T)}\xi^{3}|\psi|^2\dx\dt\right).
\end{align}

 Adding up \eqref{eq:car_init_PDE} and \eqref{eq:carODEProof1}, we can use the parameters $\lambda, s$ to absorb all the lower order terms. More precisely, we get for $\lambda,s$  sufficiently large.
\begin{align}\notag 
I(\phi)&+I(\psi;{d_v}^{-1})\\
& \leq C  \Bigg(\iint_{\omega_1\times(0,T)}e^{-2s\alpha}(s\lambda\xi)^2|\phi|^2\dx\dt\notag + s^3\lambda^4\iint_{\omega_1(t) \times(0,T)}e^{-2s\alpha} \xi^{3}|\psi|^2\dx\dt \notag\\  \label{eq:car_step2}
&\quad\quad+\iint_{Q_T}e^{-2s\alpha}\left(|g_1|^2+|{g_2}|^2\right)\dx\dt\Bigg).
\end{align}

\subsubsection*{Step 2. Local estimate for $\psi$}
%

\indent From \eqref{eq:subsetomega}, let us consider a function $\zeta\in C^\infty([0,T]\times[0,1])$ verifying 
\begin{equation*}
\begin{cases}
0\leq \zeta \leq 1 &\forall (t,x)\in [0,T]\times[0,1], \\
\zeta(t,x)=1 &\forall t\in [0,T],  \;\; \forall x\in \omega_1(t),\\
\zeta(t,x)=0 &\forall t\in[0,T], \;\; \forall x\in[0,L]\setminus\overline{\omega_2(t)}.
\end{cases}
\end{equation*}

We have
\begin{align*}
s^3\iint_{\omega_1(t)\times(0,T)}e^{-2s\alpha}\xi^3|\psi|^2\dx\dt &\leq s^3 \iint_{Q_T}e^{-2s\alpha}\xi^3\zeta|\psi|^2\dx\dt \\
&= \frac{1}{a_{21}} s^3 \iint_{Q_T}e^{-2s\alpha}\xi^3\zeta\psi(-\phi_t-a_{11}\phi-g_1)\dx\dt.
\end{align*}
Observe that at this point is crucial to have $a_{21}\neq 0$ by \eqref{eq:couplage}.

Integrating by parts in time in the right-hand side yields
\begin{align*}\notag 
&s^3\iint_{\omega_1(t)\times(0,T)}e^{-2s\alpha}\xi^3|\psi|^2\dx\dt\\
& \leq \frac{1}{a_{21}}s^3 \iint_{Q_T}(e^{-2s\alpha}\xi^3\zeta)_{t}\psi\phi\dx\dt+\frac{1}{a_{21}}s^3 \iint_{Q_T}e^{-2s\alpha}\xi^3\zeta \psi_t\phi\dx\dt \notag \\
&-\frac{a_{11}}{a_{21}}s^3\iint_{Q_T}e^{-2s\alpha}\xi^3\zeta\psi\phi\dx\dt-\frac{1}{a_{21}}s^3\iint_{Q_T}e^{-2s\alpha}\xi^3\zeta\psi {g_1}\dx\dt.
\end{align*}

Using the equation verified by $\psi$ in the second term on the right-hand side of the above equation, we get
\begin{align}\notag 
s^3\iint_{\omega_1(t)\times(0,T)}&e^{-2s\alpha}\xi^3|\psi|^2\dx\dt \\ \notag
\leq &\frac{1}{a_{21}}s^3 \iint_{Q_T}(e^{-2s\alpha}\xi^3\zeta)_{t}\psi\phi\dx\dt-\frac{{d_v}}{a_{21}} s^3 \iint_{Q_T}e^{-2s\alpha}\xi^3\zeta \psi_{xx}\phi\dx\dt \\ \notag
&-\frac{a_{12}}{a_{21}}s^3\iint_{Q_T}e^{-2s\alpha}\xi^3 \zeta |\phi|^2\dx\dt - \frac{a_{22}}{a_{21}}s^3\iint_{Q_T}e^{-2s\alpha}\xi^3 \zeta \psi\phi \dx\dt \\ \notag
&-\frac{a_{11}}{a_{21}}s^3\iint_{Q_T}e^{-2s\alpha}\xi^3\zeta\psi\phi\dx\dt-\frac{1}{a_{21}}s^3\iint_{Q_T}e^{-2s\alpha}\xi^3\zeta\psi {g_1}\dx\dt \\ \label{eq:est_local_first}
&=: \sum_{i=1}^{6} K_i.
\end{align}
We bound each term $K_i$ for $1\leq i\leq 6$. For the first one, we have
\begin{align*}
K_1=\frac{1}{a_{21}}s^3\iint_{Q_T}e^{-2s\alpha}\xi^3\zeta_t\psi\phi\dx\dt+\frac{1}{a_{21}}s^3\iint_{Q_T}(e^{-2s\alpha}\xi^3)_t\zeta \psi\phi\dx\dt.
\end{align*}
Using the properties of the function $\zeta$ and  $|(e^{-2s\alpha}\xi^{3})_t|\leq Cs^2 e^{-2s{\alpha}} \xi^{5}$, we get after applying Cauchy-Schwarz and Young inequalities that
\begin{align}\label{eq:K1_est}
|K_1|\leq 2\delta s^3\iint_{Q_T}e^{-2s\alpha}\xi^3 |\psi|^2\dx\dt+C_{\delta}s^7\iint_{\omega_2(t)\times(0,T)}e^{-2s\alpha}\xi^7 |\phi|^2\dx\dt
\end{align}
for any $\delta>0$. 

We can use definitions \eqref{eq:defminmax} and Young's inequality to obtain 
\begin{align}\notag
|K_2|& \leq {d_v}^2 s^{-2} \iint_{Q_T} e^{-2s\widehat{\alpha}} (\widehat \xi)^{-2}|\psi_{xx}|^2\dx\dt+Cs^8\iint_{\omega_2(t)\times(0,T)}e^{-4s\alpha+2s\widehat{\alpha}}(\widehat{\xi})^2 (\xi)^6 |\phi|^2\dx\dt\\
& \leq {d_v}^2 s^{-2} \iint_{Q_T} e^{-2s\widehat{\alpha}} (\widehat{\xi})^{-2}|\psi_{xx}|^2\dx\dt+Cs^8\iint_{\omega_2(t)\times(0,T)}e^{-4s\alpha^\star+2s\widehat{\alpha}}({\xi}^\star)^8|\phi|^2\dx\dt.\label{eq:K2_est}
\end{align}
Observe that the constant $C>0$ is uniform with respect to ${d_v}$ and that also we have introduced a smaller weight accompanying the variable $\psi_{xx}$. In a future step, we will estimate uniformly this new term. 

For $3\leq i\leq 6$, we can bound easily $K_i$. Using Cauchy-Schwarz and Young inequalities, a straightforward computation gives
\begin{align}\notag 
\sum_{i=3}^{6}|K_i|\leq &\  C_{\delta}\left(s^3\iint_{Q_T}e^{-2s\alpha}\xi^3|g_1|^2\dx\dt+s^3\iint_{\omega_2(t)\times(0,T)}e^{-2s\alpha}\xi^3|\phi|^2\dx\dt\right) \\ \label{eq:K3_K6_est}
&+3\delta s^3\iint_{Q_T}e^{-2s\alpha}\xi^3 |\psi|^2\dx\dt
\end{align}
for all $\delta>0$. 

Putting together \eqref{eq:car_step2}, \eqref{eq:est_local_first}, \eqref{eq:K1_est}, \eqref{eq:K2_est}, and \eqref{eq:K3_K6_est}, we can take $\delta$ sufficiently small and obtain
\begin{align}\notag 
&I(\phi)+I(\psi;{d_v}^{-1})\\ \notag
&\leq C \Bigg(s^8\iint_{\omega_2\times(0,T)}  (e^{-4s\alpha^\star+2s\widehat{\alpha}}+e^{-2s\alpha})({\xi}^\star)^{8}|\phi|^2\dx\dt \\ \notag
&\qquad \quad + {d_v}^2 s^{-2} \iint_{Q_T} e^{-2s\widehat{\alpha}} (\widehat{\xi})^{-2}|\psi_{xx}|^2 \dx\dt + \iint_{Q_T}e^{-2s\alpha} |{g_2}|^2\dx\dt \notag\\
&\qquad\quad +s^3\iint_{Q_T}e^{-2s\alpha}\xi^3|{g_1}|^2 \dx \dt \Bigg)\label{eq:car_step_3} 
\end{align}
for all $\lambda$ and $s$ sufficiently large.

\subsubsection*{Step 4. Uniform global estimate of $\psi_{xx}$ and conclusion}
We devote this step to estimate the global term of $\psi_{xx}$ appearing in \eqref{eq:car_step_3}. Notice that this integral has a factor ${d_v}^2$ so we need to estimate it uniformly with respect to the parameter ${d_v}$. 

From the PDE verified in \eqref{eq:linear_adj}, we have
\begin{align}
{d_v} & \iint_{Q_T}|\psi_{xx}|^2e^{-2s\widehat{\alpha}}(\widehat\xi)^{-2} \dx\dt \notag\\
& =-\iint_{Q_T}\psi_{t}\psi_{xx}e^{-2s\widehat{\alpha}}(\widehat{\xi})^{-2}\dx\dt-a_{12}\iint_{Q_T}\phi \psi_{xx}e^{-2s\widehat{\alpha}}(\widehat\xi)^{-2} \dx\dt \notag\\
&\quad - a_{22}\iint_{Q_T}\psi \psi_{xx}e^{-2s\widehat{\alpha}}(\widehat\xi)^{-2} \dx\dt  - \iint_{Q_T}g_2 \psi_{xx}e^{-2s\widehat{\alpha}}(\widehat\xi)^{-2} \dx\dt.\label{step4:1}
\end{align}
Integrating by parts in space on the first term on the right-hand side of \eqref{step4:1} and using Cauchy-Schwarz and Young inequalities on the other two, we readily get
\begin{align}\notag 
{d_v}  \iint_{Q_T} & |\psi_{xx}|^2e^{-2s\widehat{\alpha}}(\widehat\xi)^{-2} \dx\dt \\ \notag
&\leq \frac{1}{2}\iint_{Q_T}\left(|\psi_x|^2\right)_{t} e^{-2s\widehat{\alpha}}(\widehat\xi)^{-2}\dx\dt+ 3\delta {d_v} \iint_{Q_T} |\psi_{xx}|^2e^{-2s\widehat{\alpha}}(\widehat\xi)^{-2} \dx\dt \\ \notag
&\quad +C_{\delta}{d_v}^{-1}\iint_{Q_T}|\phi|^2e^{-2s\widehat{\alpha}}(\widehat\xi)^{-2} \dx\dt +C_{\delta}{d_v}^{-1}\iint_{Q_T}|\psi|^2e^{-2s\widehat{\alpha}}(\widehat\xi)^{-2} \dx\dt \\
& \quad +C_{\delta}{d_v}^{-1}\iint_{Q_T}|g_2|^2e^{-2s\widehat{\alpha}}(\widehat\xi)^{-2} \dx\dt  \label{eq:est_inter_unif}
\end{align}
for any $\delta>0$. Observe that we have put the parameter ${d_v}^{-1}$ in front of three of the right hand side terms, however, the first one is still missing it. Further integration by parts in the time variable and then integrating in space yields
\begin{align*}
{d_v}  \iint_{Q_T} & |\psi_{xx}|^2e^{-2s\widehat{\alpha}}(\widehat\xi)^{-2} \dx\dt \notag\\
&\leq \frac{1}{2}\iint_{Q_T} \psi \psi_{xx} \left(e^{-2s\widehat{\alpha}}(\widehat\xi)^{-2}\right)_t\dx\dt+ 3\delta {d_v} \iint_{Q_T} |\psi_{xx}|^2e^{-2s\widehat{\alpha}}(\widehat\xi)^{-2} \dx\dt \notag\\
&\quad +C_{\delta}{d_v}^{-1}\iint_{Q_T}|\phi|^2e^{-2s\widehat{\alpha}}(\widehat\xi)^{-2} \dx\dt +C_{\delta}{d_v}^{-1}\iint_{Q_T}|\psi|^2e^{-2s\widehat{\alpha}}(\widehat\xi)^{-2} \dx\dt \notag\\
& \quad +C_{\delta}{d_v}^{-1}\iint_{Q_T}|g_2|^2e^{-2s\widehat{\alpha}}(\widehat\xi)^{-2} \dx\dt, 
\end{align*}
where we have used that the weight functions are $x$-independent and $\psi$ satisfies homogeneous Neumann boundary conditions.

Arguing as we did for obtaining \eqref{eq:est_inter_unif} and using that $|(e^{-2s\widehat\alpha}(\widehat \xi)^{-2})_t|\leq Cs^2 e^{-2s\widehat{\alpha}}$, 
we get
\begin{align}
{d_v}  \iint_{Q_T} & |\psi_{xx}|^2e^{-2s\widehat{\alpha}}(\widehat\xi)^{-2} \dx\dt \notag \\
&\leq C_{\delta}{d_v}^{-1}s^4 \iint_{Q_T}|\psi|^2e^{-2s\widehat{\alpha}}(\widehat \xi)^{2}\dx\dt+ 4\delta {d_v} \iint_{Q_T} |\psi_{xx}|^2e^{-2s\widehat{\alpha}}(\widehat\xi)^{-2} \dx\dt \notag\\
&\quad +C_{\delta}{d_v}^{-1}\iint_{Q_T}|\phi|^2e^{-2s\widehat{\alpha}}(\widehat\xi)^{-2} \dx\dt +C_{\delta}{d_v}^{-1}\iint_{Q_T}|\psi|^2e^{-2s\widehat{\alpha}}(\widehat\xi)^{-2} \dx\dt \notag\\
&\quad +C_{\delta}{d_v}^{-1}\iint_{Q_T}|g_2|^2e^{-2s\widehat{\alpha}}(\widehat\xi)^{-2} \dx\dt,\label{step4:2}
\end{align}
for any $\delta>0$. Taking $\delta$ small enough and then multiplying by $s^{-2}{d_v}$ on both sides of \eqref{step4:2}, we obtain the uniform estimate
\begin{align}\notag 
{d_v}^2  s^{-2} \iint_{Q_T} & |\psi_{xx}|^2e^{-2s\widehat{\alpha}}(\widehat\xi)^{-2} \dx\dt \\ \notag
&\leq Cs^2 \iint_{Q_T}|\psi|^2e^{-2s\widehat{\alpha}}(\widehat \xi)^{2}\dx\dt+ Cs^{-2} \iint_{Q_T}|\phi|^2e^{-2s\widehat{\alpha}}(\widehat\xi)^{-2} \dx\dt \\ \label{eq:est_unif_psixx}
&\quad +Cs^{-2}\iint_{Q_T}|\psi|^2e^{-2s\widehat{\alpha}}(\widehat\xi)^{-2} \dx\dt +Cs^{-2}\iint_{Q_T}|g_2|^2e^{-2s\widehat{\alpha}}(\widehat\xi)^{-2} \dx\dt.
\end{align}

To conclude the proof, we use that $e^{-2s\widehat{\alpha}}\leq e^{-2s\alpha}$, $s^{-1}(\widehat{\xi})^{-1}\leq C$ and $\widehat{\xi} \leq \xi$, to obtain from the above estimate \eqref{eq:est_unif_psixx} that
\begin{align}\notag 
{d_v}^2  s^{-2} \iint_{Q_T}  |\psi_{xx}|^2e^{-2s\widehat{\alpha}}(\widehat\xi)^{-2} \dx\dt \leq & \ Cs^2 \iint_{Q_T}|\psi|^2e^{-2s{\alpha}} \xi^{2}\dx\dt+ C\iint_{Q_T}|\phi|^2e^{-2s{\alpha}}\dx\dt \\ \label{eq:est_unif_final}
& +C\iint_{Q_T}|\psi|^2e^{-2s{\alpha}} \dx\dt +C\iint_{Q_T}|g_2|^2e^{-2s{\alpha}} \dx\dt.
\end{align}

To conclude, we plug \eqref{eq:est_unif_final} into \eqref{eq:car_step_3} and employ the parameter $s$ to absorb the remaining global terms. The result follows by using \eqref{eq:defminmax} to show that $e^{-2 s \alpha} \leq e^{-4s \alpha^\star + 2 s \widehat{\alpha}}$.
\end{proof}

We are going to improve the Carleman inequality \eqref{eq:car_Prop} in the sense that the new weight functions, that we will use, will only vanish as $t\to T^-$. To this end, let us consider the function 
\begin{equation*}
\ell(t)=\begin{cases}
1 &\text{for } 0\leq t \leq T/2, \\
r(t) &\text{for } T/2 \leq t \leq T,
\end{cases}
\end{equation*}
where we recall that $r(t)$ is defined in \eqref{eq:def_r}. We introduce the new weight functions
\begin{equation}\label{eq:weight_cut}
\begin{split}
&\beta(x,t):=\ell(t)(e^{2\lambda\|\eta\|_\infty}-e^{\lambda\eta(x,t)}), \quad \gamma(x,t):=\ell(t)e^{\lambda\eta(x,t)}, \\
&\widehat \beta(t):=\max_{x\in[0,1]} \beta(x,t), \quad \beta^\star(t):=\min_{x\in[0,1]} \beta(x,t), \\
&\widehat \gamma(t):=\min_{x\in[0,1]} \gamma(x,t), \quad \gamma^\star(t):=\max_{x\in[0,1]} \gamma(x,t).
\end{split}
\end{equation}
Observe that in this case, the corresponding weight functions only blow up as $t\to T^{-}$.

We deduce from \Cref{prop:UniformCarleman} and energy estimates the following Carleman estimate.
\begin{prop}
\label{prop:UniformCarlemanTer}
There exist $\lambda, s$ sufficiently large and a positive constant $C = C(s,\lambda,T)>0$ such that for any given ${d_v}\geq 1$, $(\phi_T,\psi_T) \in L^2(\Omega)^2$, the solution to \eqref{eq:linear_adj} verifies
\begin{align}\notag 
&\| \phi(0)\|^2_{L^2(\Omega)}  +\| \psi(0)\|^2_{L^2(\Omega)} + \iint_{Q_T} e^{-2s\widehat{\beta}} \widehat{\gamma} | \phi|^2  \dx\dt + \iint_{Q_T} e^{-2s\widehat{\beta}} (\widehat{\gamma})^3 | \psi|^2 \dx\dt \\ 
& \leq  C\Bigg(\iint_{Q_T} e^{-2s\beta^\star} (\gamma^\star)^3 | g_1|^2 \dx\dt+\iint_{Q_T} e^{-2s\beta^\star} | g_2|^2\dx\dt\notag\\
&\quad\quad+\iint_{\omega_2(t)\times(0,T)}  e^{-s\beta^\star}({\gamma}^\star)^{8}| \phi|^2\dx\dt \Bigg). \label{eq:car_Prop2}
\end{align}
\end{prop}
The proof of this result can be found in Appendix \ref{app:obs}. Increasing the constant $C$ in \eqref{eq:car_Prop2} if necessary, we can assume that 
\begin{equation}
\label{eq:bstarbhat}
\forall t \in (0,T),\ \widehat{\beta}(t) < \frac{3}{2} \beta^{\star}(t).
\end{equation}

\subsection{Null-controllability despite a source term}
\label{sec:nullsource}
Now we proceed to the definition of the spaces where the linear system \eqref{eq:linearized_diffusion} will be solved. We define the differential operator
\begin{equation}
\label{eq:defL}
L(U,V) = (\partial_t U - a_{11} U - a_{12} V, \partial_t V - {d_v} \partial_{xx} V - a_{21} U - a_{22} V),
\end{equation}
and define the space 
\begin{equation}
\label{eq:defW}
W(0,T):= L^\infty(0,T;H^1(\Omega))\cap H^1(0,T;L^2(\Omega)).
\end{equation}

Let us set
\begin{align}
\notag E := \Bigg\{ (U,V,h) \in E_0 : \ &e^{s \widehat{\beta} } ( \widehat{\gamma})^{-1/2} \left((L(U,V))_{1}- h  \mathbf{1}_{\omega(t)}\right) \in L^2(Q_T),\\
& e^{s \widehat{\beta} } ( \widehat{\gamma})^{-3/2} (L(U,V))_{2} \in L^2(Q_T)\quad \text{and}\ \partial_{x} V = 0\ \text{on}\ \Sigma_T\Bigg\},\label{eq:defE}
\end{align}
where
\begin{align*}
E_0 := \Bigg\{ &(U,V,h) : \ e^{s \beta^\star} (\gamma^\star)^{-3/2} U \in L^2(Q_T),\ e^{s \beta^\star}  V \in L^2(Q_T), \ e^{(s/2) \beta^\star} ( \gamma^\star)^{-4}  \mathbf{1}_{\omega(t)} h \in L^2(Q_T)\notag\\
& \quad \ e^{s \widehat{\beta} - (s/2) \beta^*} (\widehat{\gamma})^{-1/4} (U,V) \in H^1(0,T;L^2(\Omega)) \times W(0,T),\\
& \quad  e^{(s/2)\beta^*} (\widehat{\gamma})^{-1/4} (U,V) \in H^1(0,T;L^2(\Omega)) \times W(0,T)\Bigg\}. 
\end{align*}
We endow $E$ with the following norm 
\begin{align*}
\norme{(U,V,h)}_{E} &:= \norme{e^{s \beta^\star} (\gamma^\star)^{-3/2} U}_{L^2(Q_T)} + \norme{e^{s \beta^\star}  V}_{L^2(Q_T)}+  \norme{e^{(s/2) \beta^\star} ( \gamma^\star)^{-4}h  \mathbf{1}_{\omega(t)}}_{L^2(Q_T)}\\
& + \norme{e^{s \widehat{\beta} - (s/2) \beta^\star}  (\widehat{\gamma})^{-1/4} (U,V)}_{H^1(0,T;L^2(\Omega)) \times W(0,T)}\\
& + \norme{e^{(s/2)\beta^\star} (\widehat{\gamma})^{-1/4} (U,V)}_{H^1(0,T;L^2(\Omega)) \times W(0,T)}\\
& + \norme{e^{s\widehat{\beta}} (\widehat{\gamma})^{-1/2} \left((L(U,V))_{1}- h  \mathbf{1}_{\omega(t)}\right)}_{L^2(Q_T)} + \norme{e^{s \widehat{\beta}} (\widehat{\gamma})^{-3/2} (L(U,V))_{2}}_{L^2(Q_T)},
\end{align*}
that makes $(E, \norme{\cdot}_{E})$ a Banach space.\\
\indent We also introduce 
\begin{align}
X &= \{(F_1,F_2) : e^{s \widehat{\beta}} (\widehat{\gamma})^{-1/2} F_1 \in L^2(Q_T),\ e^{s \widehat{\beta}} (\widehat{\gamma})^{-3/2} F_2 \in L^2(Q_T)\},\notag\\
Y &= L^2(\Omega)\times H^1(\Omega),\notag\\
G &= X \times Y,\label{eq:defG}
\end{align}
endowed with their natural norm.

The goal of this part is to prove the following result.
\begin{prop}
\label{prop:NullControllabilityLinear}
For every $(F_1,F_2, U_0,V_0) \in G$, there exists a control $h \in L^2(0,T;L^2(\Omega))$, bounded independently of ${d_v}$, such that, if $(U,V)$ is the associated solution to \eqref{eq:linearized_diffusion}, one has $(U,V,h) \in E$. In particular, $(U,V) (T) = 0$ holds.

Moreover, there exists a positive constant $C>0$ independent of ${d_v}$ such that
\begin{align}
\norme{(U,V,h)}_{E} \leq C \norme{(F_1,F_2,U_0,V_0)}_{G}.
\label{eq:EstimationSol}
\end{align}
\end{prop}
\begin{proof}
Let $L^*$ be the adjoint operator of $L$, defined in \eqref{eq:defL},
\begin{equation*}
L^*(\phi,\psi) = (-\partial_t \phi - a_{11} \phi - a_{21} \psi, -\partial_t \psi - {d_v} \partial_{xx} \psi - a_{12} \phi - a_{22} \psi),
\end{equation*}
and let us introduce the space
\begin{equation*}
P_0 := \{ (\phi, \psi) \in C^{\infty}(\overline{Q_T})^2\ ;\ \partial_{x} \psi = 0\ \text{on}\ \Sigma_T\}.
\end{equation*}
We now define the following bilinear form, for $(\phi_1, \psi_1)$, $(\phi_2, \psi_2)$, as
\begin{align*}
a((\phi_1, \psi_1), (\phi_2, \psi_2)) &:= \iint e^{-2s \beta^\star} (\gamma^\star)^3\Bigg((L^*(\phi_1, \psi_1))_1 (L^*(\phi_2, \psi_2))_1 \Bigg)\dx\dt \\
& + \iint e^{-2s \beta^\star} (L^*(\phi_1, \psi_1))_2 L^*(\phi_2, \psi_2))_2 \dx\dt \\
& +  \iint e^{-s \beta^\star } ( \gamma^\star)^8  \mathbf{1}_{\omega(t)}^2 \phi_1 \phi_2 \dx\dt .
\end{align*}
\indent The bilinear symmetric positive form $a$ on $P_0$ is definite. Indeed if $a((\phi,\psi), (\phi,\psi)) = 0$ then the right hand side of the Carleman estimate \eqref{eq:car_Prop2} is equal to $0$ then $\phi = \psi  \equiv 0$. So $a(\cdot, \cdot)$ is a scalar product on $P_0$. Therefore, we can consider the space $P$, the completion of $P_0$ with respect to the norm associated to the scalar product defined by $a$, denoted by $\norme{\cdot}_P$. This makes $(P, \norme{\cdot}_P)$ a Hilbert Space and $a(\cdot, \cdot)$ is a coercive, continuous, bilinear form on $P$.\\
\indent We now introduce the linear form $l$, for $(\phi, \psi) \in P$, 
\begin{align*}
l\left((\phi, \psi)\right) &= \iint_{Q_T} F_1 \phi + \iint_{Q_T} F_2 \psi  \dx\dt + \int_{\Omega} U_0(x) \phi(0,x) + \int_{\Omega} V_0(x) \psi(0,x) \dx
\end{align*}
It is easy to show that $l$ is a continuous linear form on $P$ thanks to the Carleman estimate \eqref{eq:car_Prop2},
\begin{align}
\label{eq:continuityl}
|l\left((\phi, \psi)\right)|  \leq C \norme{(F_1,F_2,U_0,V_0)}_{G}\norme{(\phi, \psi)}_{P}. 
\end{align}

Consequently, by using Lax-Milgram's lemma, there exists a unique $(\tilde{\phi}, \tilde{\psi}) \in P$ satisfying
\begin{equation}
\label{eq:formvara}
\forall (\phi, \psi) \in P, \quad a((\tilde{\phi}, \tilde{\psi}), (\phi,\psi)) = l((\phi,\psi)).
\end{equation}
We set
\begin{equation}
\label{eq:setUVh}
(\tilde{U}, \tilde{V}) = (e^{-2 s \beta^\star} (\gamma^\star)^{3} (L^*(\tilde{\phi}, \tilde{\psi}))_1, e^{-2 s \beta^\star} (L^*(\tilde{\phi}, \tilde{\psi}))_2)\ \text{and}\ \tilde{h} = -e^{-s \beta^\star} (\gamma^\star)^{8}  \mathbf{1}_{\omega(t)} \tilde{\phi}.
\end{equation}
We easily deduce from \eqref{eq:continuityl}, \eqref{eq:formvara} with $(\phi, \psi) = (\tilde{\phi}, \tilde{\psi})$ and \eqref{eq:setUVh} that
\begin{align}
&\norme{e^{s \beta^\star} ( \gamma^\star)^{-3/2} \tilde{U}}_{L^2(Q_T)} + \norme{e^{s \beta^\star}  \tilde{V}}_{L^2(Q_T)} +  \norme{e^{(s/2) \beta^\star } ( \gamma^\star)^{-4}  \mathbf{1}_{\omega(t)} h}_{L^2(Q_T)}\notag \\
& \leq C \norme{(F_1,F_2,U_0,V_0)}_{G}\label{eq:EstimationLaxMilgram}
\end{align}
\indent Let $(U,V)$ be the weak solution to
\begin{equation}\label{eq:linearized_diffusionProof}
\begin{cases}
\partial_t U=a_{11} U + a_{12} V + F_1 +   \tilde{h} \mathbf{1}_{\omega(t)} & \text{in } Q_T, \\
\partial_t V- {d_v}\partial_{xx}V  = a_{21} U + a_{22} V + F_2 &\text{in } Q_T, \\
\partial_x V=0 &\text{on } \Sigma_T, \\
(U,V)(0,\cdot)=(U_0,V_0) &\text{in }(0,1),
\end{cases}
\end{equation} 
This means that $(U,V)$ is the solution of \eqref{eq:linearized_diffusionProof} defined by transposition i.e. $(U,V)$ is the unique function satisfying 
\begin{align}
&\iint_{Q_T} (U, V) \cdot (g_1, g_2) \dx\dt \notag\\
 & = \iint_{Q_T} F_1 \phi  + F_2 \psi  \dx\dt +\iint_{Q_T} \tilde{h} \phi \dx\dt + \int_{\Omega} U_0(x) \phi(0,x) + V(0,x) \psi(0,x) \dx,
\label{eq:transposition}
\end{align}
for every $(g_1, g_2)$ where $(\phi, \psi)$ is the solution to the adjoint system \eqref{eq:linear_adj} with $\phi_T = \psi_T = 0$. But from the variational formulation \eqref{eq:formvara} satisfied by $(\tilde{U}, \tilde{V})$, we have that $(\tilde{U}, \tilde{V})$ also satisfies \eqref{eq:transposition} then by uniqueness we get 
\begin{equation}
\label{eq:tildeUU}
(\tilde{U}, \tilde{V}) = (U,V).
\end{equation}
\indent It remains to prove the fact that $(U,V,h) \in E$. For some function depending on time $\rho(t)= e^{s \widehat{\beta} - (s/2) \beta^*}(\widehat{\gamma})^{-1/4}\ \text{or}\ e^{ (s/2) \beta^*} (\widehat{\gamma})^{-1/4} $, we introduce 
\begin{equation}
\label{eq:defUstar}
(U^*, V^*) = \rho(t) (U,V).
\end{equation}
From \eqref{eq:linearized_diffusionProof}, an easy computation shows that $(U^*, V^*)$ is the solution to the following system
\begin{equation}\label{eq:linearized_diffusion_star}
\begin{cases}
\partial_t U^*=a_{11} U^* + a_{12} V^* + \rho F_1 +  \rho h  \mathbf{1}_{\omega(t)}  + \rho_t U& \text{in } Q_T, \\
\partial_t V^*- {d_v}\partial_{xx}V^*  = a_{21} U^* + a_{22} V^* +  \rho F_2+ \rho_t V &\text{in }Q_T, \\
\partial_x V^*=0 &\text{on } \Sigma_T, \\
(U^*,V^*)(0,\cdot)=\rho(0)(U_0,V_0) &\text{in }(0,1),
\end{cases}
\end{equation}
The goal is to prove that there exists a positive constant $C>0$ such that
\begin{align}
\label{eq:EstimationUVE}
\norme{U^*}_{H^1(0,T;L^2(\Omega))} + \norme{V^*}_{W(0,T)} 
\leq C \norme{(F_1,F_2,U_0,V_0)}_{G}.
\end{align}
\indent For $\rho = e^{s \widehat{\beta} - (s/2) \beta^*}(\widehat{\gamma})^{-1/4}$ or $\rho=e^{ (s/2) \beta^*} (\widehat{\gamma})^{-1/4} $, it is easy to show from \eqref{eq:bstarbhat} that there exists a positive constant $C>0$ such that
\begin{equation}
\label{eq:rhostar}
|\rho_t|  \leq  C e^{s \beta^*}  (\gamma^*)^{-3/2}\leq C e^{s \beta^*}.
\end{equation}
So, from \eqref{eq:rhostar} and \eqref{eq:EstimationLaxMilgram}, we deduce that the right hand sides of the two first equations of \eqref{eq:linearized_diffusion_star} belong to $L^2$. Then by parabolic regularity in $L^2$ for the second equation, using the fact that $V_0 \in H^1(\Omega)$, we obtain \eqref{eq:EstimationUVE}.\\
\indent From \eqref{eq:EstimationLaxMilgram}, \eqref{eq:defUstar}, \eqref{eq:EstimationUVE}, \eqref{eq:linearized_diffusionProof} and \eqref{eq:tildeUU} we have that $(U,V,h) \in E$ and satisfies the desired estimate \eqref{eq:EstimationSol}. This concludes the proof.
\end{proof}

\section{Proof of the local-controllability results}

\subsection{A precise inverse mapping argument for the ODE-PDE system}
\label{sec:inversemappingarg}

The goal of this section is to prove \Cref{th:mainresult2}.\\
\indent Recalling the change of variables performed in \Cref{sec:nulllin}, we remark that we have reduced our local-controllability problem around $(u_{\pm}, v_{\mp})$ for \eqref{eq:semilinear_diffusion} to a local null-controllability problem for the variable $(U, V)$ satisfying \eqref{eq:nonlinear_diffusionUV}. So, we will use the null-controllability result for the linearized system \eqref{eq:linearized_diffusion}, established in \Cref{prop:NullControllabilityLinear} and a precise inverse mapping argument, see \Cref{th:inversionlocal} below. We emphasize that the regularity assumptions on the trajectory of the linear system, contained in the definition of the space $E$, are crucial to treat the nonlinear term.
\begin{thm}[See \cite{Don96}]
\label{th:inversionlocal}
Let $E$ and $G$ be two Banach spaces and let $\mathcal{A} \in C^1( E ; G)$ with $\mathcal{A}(0) = 0$. We assume that $\mathcal{A}'(0)$ is an isomorphism from $E$ onto $G$. More precisely, we assume that there exists $C_0 >0$ such that
\begin{equation}
\label{eq:ConstantInjectivity}
\norme{e}_{E} \leq C_0 \norme{\mathcal{A}'(0)(e)}_{G},\  \forall e \in E,
\end{equation}
and that there exists $0<\delta < C_0^{-1}$ and $\eta >0$,
\begin{equation}
\label{eq:diffA}
\norme{\mathcal{A}(e_1) - \mathcal{A}(e_2) - \mathcal{A}'(0)(e_1-e_2)} \leq \delta \norme{e_1-e_2},\  \forall e_1, e_2 \in B_{\eta}(0).
\end{equation}
Then,  the equation $\mathcal{A}(e) = g$ has a solution $e \in B_{\eta}(0)$ for all $\norme{g}_{G} \leq (C_0^{-1} - \delta)^{-1} \eta$.
\end{thm}
\begin{rmk}\label{rmk:modulecontinuity}
By the mean value theorem, it can be shown that for any $0<\delta < C_0^{-1}$, inequality \eqref{eq:diffA} is satisfied for $\eta$ such that 
\begin{equation*}
\norme{\mathcal{A}'(e) - \mathcal{A}'(0)} \leq \delta,\ \forall \norme{e} \leq \eta.
\end{equation*}
\end{rmk}
In our setting, we use the previous theorem with the space $E$ defined in \eqref{eq:defE} and $G$ defined in \eqref{eq:defG} and the operator
\begin{equation*}
\mathcal{A}(U, V, h) = (L(U,V) + N(U,V)+ ( - h  \mathbf{1}_{\omega(t)},0), (U(0, \cdot), V(0, \cdot)), \forall (U,V,h) \in E.
\end{equation*}
Recall that the linear term $L$ is defined in \eqref{eq:defL} and the nonlinear term $N$ is defined in \eqref{eq:NonlinearN}. We have 
\begin{equation*}
\mathcal{A}'(0,0,0)\cdot (U,V,h) = (L(U,V) + ( - h  \mathbf{1}_{\omega(t)},0), U(0,\cdot), V(0, \cdot))\ \forall (U,V,h) \in E.
\end{equation*}
We have the following regularity result for the operator $\mathcal{A}$.
\begin{prop}
\label{prop:regA}
We have that $\mathcal{A} \in C^1(E;G)$.
\end{prop}
\begin{proof}
First, we remark that the linear terms in the definition of $\mathcal{A}$ are continuous then continuously differentiable because of the definition of the space $E$.

Let us show that the nonlinear term $N(U,V) \in C^1(E;G)$. For increasing positive constants $\overline{C} > 0$, we have that
\begin{align*}
\norme{N(U,V)}_{G} &\leq \overline{C}  \norme{e^{s \widehat{\beta}} (\widehat{\gamma})^{-1/2}  (UV^2 + 2 v_{\mp} UV + U_{\pm} V^2)}_{L^2(Q_T)}\\
 & \leq \overline{C} \left( \norme{ e^{s \widehat{\beta}} (\widehat{\gamma})^{-1/2} U V^2}_{L^2(Q_T)} + \norme{ e^{s \widehat{\beta}} (\widehat{\gamma})^{-1/2} U V}_{L^2(Q_T)}  + \norme{ e^{s \widehat{\beta}} (\widehat{\gamma})^{-1/2} V^2}_{L^2(Q_T)} \right).
 \end{align*}
\indent For the cubic term, we have
\begin{align*}
\norme{ e^{s \widehat{\beta}} (\widehat{\gamma})^{-1/2} U V^2}_{L^2(Q_T)} & \leq \overline{C} \norme{ e^{s \widehat{\beta} -(s/2)\beta^\star} (\widehat{\gamma})^{-1/4} U}_{L^{2}(Q_T)} \norme{ e^{(s/2)\beta^\star} (\widehat{\gamma})^{-1/4} V^2}_{L^{\infty}(Q_T)}\\
& \leq \overline{C} \norme{ e^{s \widehat{\beta} -(s/2)\beta^\star} (\widehat{\gamma})^{-1/4} U}_{L^{2}(Q_T)} \norme{ e^{(s/2)\beta^\star} (\widehat{\gamma})^{-1/4} V}_{L^{\infty}(Q_T)}^2\\
& \leq \overline{C} \norme{ e^{s \widehat{\beta} -(s/2)\beta^\star} (\widehat{\gamma})^{-1/4} U}_{L^{2}(Q_T)} \norme{ e^{(s/2)\beta^\star} (\widehat{\gamma})^{-1/4} V}_{W(0,T)}^2\\
& \leq \overline{C} \norme{(U,V,h)}_{E}^3.
\end{align*}
%
Note that we have used the embedding $W(0,T) \hookrightarrow L^{\infty}(0,T)$, see \eqref{eq:defW} for the defintion of $W(0,T)$.\\
\indent For the two quadratic terms, we easily have, using same type of arguments,
\begin{equation*}
 \norme{ e^{s \widehat{\beta}} (\widehat{\gamma})^{-1/2} U V}_{L^2(Q_T)}  + \norme{ e^{s \widehat{\beta}} (\widehat{\gamma})^{-1/2} V^2}_{L^2(Q_T)} \leq \overline{C} \norme{(U,V,h)}_{E}^2.
\end{equation*}
 This proves that $N(U,V) \in C^1(E;G)$ and concludes the proof.
\end{proof}
\begin{proof}[Proof of \Cref{th:mainresult1}]
An application of \Cref{th:inversionlocal} gives the existence of $\delta, \eta>0$, which a priori depend on ${d_v}$, such that if $$\norme{(u_0-u_{\pm}, v_0-v_{\pm})}_{L^2(\Omega) \times H^1(\Omega)} \leq (C_0^{-1} - \delta)^{-1}\eta,$$ there exists a control $h = h({d_v})$ such that the associated solution $(U,V)$ to \eqref{eq:nonlinear_diffusionUV} verifies $$(U,V)(T) =0\ \text{and}\ \norme{(U,V,h)}_{E} \leq \eta.$$
\indent To finish the proof of \Cref{th:mainresult2}, we must show that the constants $C_0$, $\eta$ and $\delta$ do not depend on ${d_v}$. This is actually a direct consequence from the fact that the constant $C_0$ in \eqref{eq:ConstantInjectivity}, which is actually the constant $C$ appearing in \Cref{prop:NullControllabilityLinear}, does not depend on ${d_v}$. So from \Cref{rmk:modulecontinuity} we can take $\delta < C_0^{-1}$ and $\eta$ can be chosen smaller than $\delta/\overline{C}$, $\sqrt{\delta/\overline{C}}$, where $\overline{C}$ is the maximal constant appearing in the proof of \Cref{prop:regA}.\\
\indent The expected bound \eqref{eq:boundedsigmauvh} follows from $\norme{(U,V,h)}_{E} \leq \eta$ and the definition of the space $E$. 
\end{proof}

\subsection{The shadow limit ${d_v} \rightarrow + \infty$ to reduce to the ODE-ODE system}

The goal of this section is to prove \Cref{th:mainresult1}. In order to do this, we will use \Cref{th:mainresult2} and the following result, which deals with the convergence of the system \eqref{eq:semilinear_diffusion} to the system \eqref{eq:nonlinear_shadow_SC} in the limit ${d_v} \rightarrow + \infty$.

\begin{prop}
\label{prop:ConvSol}
Let $(u_{{d_v}}, v_{{d_v}}, h_{{d_v}})\in H^1(0,T;L^2(\Omega)) \times L^{\infty}(Q_T) \times L^{2}(Q_T)$ be a solution of \eqref{eq:semilinear_diffusion} associated to $(u_0, v_0) \in L^2(\Omega)\times H^1(\Omega)$ such that
\begin{align}
&u_{{d_v}} \rightharpoonup u\ \text{in}\ H^1(0,T;L^2(\Omega))\ \text{as}\ {d_v} \rightarrow + \infty,\label{eq:convu}\\
&v_{{d_v}} \rightarrow v\ \text{in}\ L^{\infty}(Q_T)\ \text{as}\ {d_v} \rightarrow + \infty,\label{eq:convv}\\
&h_{{d_v}}\rightharpoonup h\ \text{in}\ L^{2}(Q_T)\ \text{as}\ {d_v} \rightarrow + \infty.\label{eq:convh}
\end{align}
Then, $(u,w):=(u,v)$ is the solution to \eqref{eq:nonlinear_shadow} associated to $h$ and $(u(0,\cdot),w(0))=(u_0, \int_{\Omega} v_0)$.
\end{prop}

\begin{proof}
For the purpose of the proof, let us denote
$$ f(u,v) = -uv^2 + F(1-u),\ g(u,v) = uv^2 - (F+k) v.$$  

Let $(\tilde{u}, \tilde{v})$ be the solution to 
\begin{equation}\label{eq:solutionIntTilde}
\begin{cases}
\partial_t \tilde{u}=f( u, v) + h \mathbf{1}_{\omega(t)} & \text{in } Q_T, \\
\D  \tilde {v}^\prime=\int_{\Omega} g( u, v)\dx &\text{in }(0,T), \\
\D \tilde{u}(0,x)=u_0(x)\ \text{in }\Omega, \quad  \tilde{v}(0)=\int_{\Omega} v_0,
\end{cases}
\end{equation}
where $(u,v,h)$ are the limits coming from \eqref{eq:convu}, \eqref{eq:convv}, \eqref{eq:convh}.

Taking the difference between the first equations of systems \eqref{eq:semilinear_diffusion} and \eqref{eq:solutionIntTilde}, we easily obtain that 
\begin{align*}
\partial_t u_{d_v} - \partial_t u &=f(u_{{d_v}}, v_{{d_v}}) - f(u,v) + (h_{d_v} - h) 1_{\omega(t)}\\
 &= - v_{d_v}^2u_{d_v} + F(1-u_{d_v})+ v^2u-F (1-u)+ (h_{d_v} - h) 1_{\omega(t)}\\
& = - v^2(u_{d_v}-u)- u_{d_v}(v_{d_v}^2-v^2)-F(u_{{d_v}}-u)+ (h_{d_v} - h) 1_{\omega(t)}.
\end{align*}
By using \eqref{eq:convu}, \eqref{eq:convv} and \eqref{eq:convh}, it is straightforward to see that 
\begin{equation}
\label{eq:convpartialtu}
 \partial_t u_{d_v}- \partial_t \tilde{u}  \rightharpoonup 0 \quad\text{in } L^2(Q_T)\ \text{as}\ {d_v} \rightarrow + \infty.
\end{equation}
Then, by writing for every $\chi \in L^2(Q_T)$, using Fubini's theorem and \eqref{eq:convpartialtu}, we have
\begin{align*} 
\int_{Q_T} (u_{d_v}(t,x) - \tilde{u}(t,x)) \chi(t,x) \d{t} \d{x} &= \int_0^T \int_0^1 \left(\int_0^t (\partial_t u_{d_v}- \partial_t \tilde{u})(s,x) \d{s} \right) \chi(t,x)  \d{x} \d{t}\\
& = \int_0^T \int_0^1 (\partial_t u_{d_v}- \partial_t \tilde{u})(s,x)\left( \int_s^T \chi(t,x) \d{t} \right) \d{s}\d{x}\\
\int_{Q_T}(u_{d_v}(t,x) - \tilde{u}(t,x))  \chi(t,x) \d{t} \d{x}& \rightharpoonup 0 \ \text{as}\ {d_v} \rightarrow + \infty,
\end{align*}
so
\begin{equation}\label{eq:conv_u_tilde}
u_{d_v}-\tilde{u}  \rightharpoonup 0 \quad\text{in } H^1(0,T;L^2(\Omega))\ \text{as}\ {d_v} \rightarrow + \infty.
\end{equation}
Therefore, by uniqueness, recalling \eqref{eq:convu} and \eqref{eq:conv_u_tilde}, we have 
\begin{equation}
\label{eq:equalityuutilde}
u = \tilde{u}.
\end{equation}

For the second equation, we begin by writing the solution of \eqref{eq:solutionIntTilde} and \eqref{eq:semilinear_diffusion} as follows
\begin{align}
\label{eq:explicitv}&\tilde v(t)=\int_{\Omega}v_0+\int_0^t\int_\Omega g( u(x,s), v(s))\dx \d{s}, \\
\label{eq:explicitvsigma}&v_{d_v}(t)=e^{t{d_v}\partial_{xx}}v_0+\int_0^{t}e^{(t-s){d_v}\partial_{xx}}g(u_{d_v}(s),v_{d_v}(s))\d{s}.
\end{align}
Taking the difference between \eqref{eq:explicitvsigma} and \eqref{eq:explicitv} and computing the $L^2$-norm, we get
\begin{align*}\notag 
\|v_{d_v}(t)-\tilde v(t)\|_{L^2(\Omega)}\leq & \norme{e^{t{d_v}\partial_{xx}}\left(v_0-\int_{\Omega}v_0\dx\right)}_{L^2(\Omega)} \\
&+\norme{\int_{0}^{t}e^{(t-s){d_v}\partial_{xx}}\left(g(u_{d_v}(s),v_{d_v}(s))-\int_\Omega g( u(s), v(s))\dx\right)\d{s}}_{L^2(\Omega)},
\end{align*}
where $\{e^{t{d_v}\partial_{xx}}\}_{t\geq 0}$ stands for the heat semigroup associated to the diffusion parameter $d_v >0 $ with homogeneous Neumann boundary conditions on the interval $(0,1)$.
Employing property $a.$ of \Cref{lem:semi_properties} with $K=\int_{\Omega}g(u_{d_v}(s),v_{d_v}(s)) - g(u(s),v(s))\dx$, we can easily deduce that
\begin{align}\notag 
\|v_{d_v}(t)-\tilde v(t)\|_{L^2(\Omega)}\leq & \norme{e^{t{d_v}\partial_{xx}}\left(v_0-\int_{\Omega}v_0\dx\right)}_{L^2(\Omega)} \\ \notag
&+\norme{\int_{0}^{t}e^{(t-s){d_v}\partial_{xx}}\left(g(u_{d_v}(s),v_{d_v}(s))-\int_\Omega g(u_{d_v}(s), v_{d_v}(s))\dx\right)\d{s} }_{L^2(\Omega)} \\ \label{eq:est_diff_v_vtilde}
&+\left| \int_0^{t}\int_{\Omega}g(u_{d_v}(s),v_{d_v}(s))-g( u(s), v(s))\dx\d{s} \right|.
\end{align}
Let us treat the three terms in the right hand side of \eqref{eq:est_diff_v_vtilde} separately.\\
\indent For the first one, we can use property $b.$ in \Cref{lem:semi_properties} with $z_0=v_0$ to obtain 
\begin{align}\label{eq:first_term}
t^{1/2}\norme{e^{t{d_v}\partial_{xx}}\left(v_0-\int_{\Omega}v_0\dx\right)}_{L^2(\Omega)}\leq C t^{1/2} e^{-\lambda_1 d_v t} &\leq C (d_v)^{-1/2} (d_v t)^{1/2} e^{-\lambda_1 d_v t}\notag\\
 & \leq C{d_v}^{-1/2} \rightarrow 0,\ \text{as}\ {d_v}\to +\infty.
\end{align}

For the second term, noting that it has zero-mean and that $g(u_{d_v},v_{d_v})-\int_{\Omega}g(u_{d_v},v_{d_v})\dx$ is bounded in $L^2$, using again $b.$ in \Cref{lem:semi_properties}, \eqref{eq:estisemigroupNeumann}, we get
\begin{align}\notag 
& \norme{\int_0^t e^{(t-s) {d_v} \partial_{xx}} \left(g(u_{{d_v}}(s), v_{{d_v}}(s)) -  \int_{\Omega} g(u_{{d_v}}(s), v_{{d_v}}(s)) \d{s}\right)}_{L^{2}(\Omega)} \\
& \leq C \int_0^t e^{-\lambda_1 {d_v} (t-s)} ds \leq C \frac{1}{\lambda_1 {d_v}} (1 - e^{-\lambda_1 {d_v} t}) \leq \frac{C}{\lambda_1 {d_v}} \rightarrow 0\ \text{ as }\ {d_v} \rightarrow + \infty \label{eq:est_second_term}
\end{align}
\indent For the last one, we write
\begin{equation*}
\left| \int_0^{t}\int_{\Omega}g(u_{d_v}(s),v_{d_v}(s))-g( u(s), v(s))\right| = \left|\int_{0}^{t}\int_{\Omega}v_{d_v}^2u_{d_v}-(F+k)v_{d_v}-v^2u+(F+k)v\right|
\end{equation*}
Adding and subtracting $v^2u_{d_v}$ and rearranging terms, we get
\begin{align}
&\left| \int_0^{t}\int_{\Omega}g(u_{d_v}(s),v_{d_v}(s))-g( u(s), v(s))\right|\notag\\
&=\left|\int_{0}^{t}\int_{\Omega}v^2(u_{d_v}-u)+\int_{0}^{t}\int_\Omega u_{d_v}(v_{d_v}^2-v^2)+(F+k)\int_{0}^{t}\int_{\Omega} (v-v_{d_v})\right| \label{eq:lasttermthree}
\end{align}
We use triangle inequality and see the behavior of each term in the right hand side of \eqref{eq:lasttermthree}. For the first one, by using \eqref{eq:convu} and \eqref{eq:convv}, we have
\begin{equation}
\lim_{{d_v}\to \infty}\left|\int_{0}^{t}\int_{\Omega}v^2(u_{d_v}-u)\right|=\left|\lim_{{d_v}\to \infty}\int_{0}^{t}\int_{\Omega}v^2(u_{d_v}-u)\right|=0. \label{eq:lastermuno}
\end{equation}
For the second one, using again \eqref{eq:convu} and \eqref{eq:convv} we have
\begin{align}
\left |\int_{0}^t\int_{\Omega}u_{d_v}(v_{d_v}^2-v^2)\right|\leq \|u_{d_v}\|_{L^2(Q_T)}\|v_{d_v}^2-v^2\|_{L^2(Q_T)}  \leq C\|v_{d_v}-v\|_{L^\infty(Q_T)} \rightarrow 0.\label{eq:lasttermbis}
\end{align}
For the last one, from \eqref{eq:convu}, we have 
\begin{equation}
\label{eq:lastermthreeter}
\left|(F+k)\int_{0}^{t}\int_{\Omega} (v-v_{d_v})\right| \rightarrow 0\ \text{as}\ {d_v} \rightarrow + \infty.
\end{equation}
From \eqref{eq:lasttermthree}, \eqref{eq:lastermuno}, \eqref{eq:lasttermbis}, \eqref{eq:lastermthreeter}, we deduce that
\begin{equation}
\left| \int_0^{t}\int_{\Omega}g(u_{d_v}(s),v_{d_v}(s))-g( u(s), v(s))\right| \rightarrow 0\ \text{as}\ {d_v} \rightarrow + \infty.
\label{eq:est_three_term}
\end{equation}
\indent So from \eqref{eq:first_term}, \eqref{eq:est_second_term} and \eqref{eq:est_three_term}, we obtain that
\begin{equation}\label{eq:conv_delta}
\|v_{d_v}-\tilde v\|_{C([\delta,T];L^2(\Omega))} \to 0 \quad \text{ as } {d_v}\to+\infty,
\end{equation}
for any $\delta>0$.
So, by uniqueness, recalling \eqref{eq:convv} and \eqref{eq:conv_delta}, we have 
\begin{equation}
\label{eq:equalityvvtilde}
v = \tilde{v}.
\end{equation}
Hence, from \eqref{eq:solutionIntTilde}, \eqref{eq:equalityvvtilde} and \eqref{eq:equalityuutilde}, we obtain the conclusion of the proof.
\end{proof}
Now, we prove \Cref{th:mainresult1}.
\begin{proof}[Proof of \Cref{th:mainresult1}]
From \Cref{th:mainresult2}, we deduce that one can find $h_{{d_v}}$ such that \eqref{eq:boundedsigmauvh} and \eqref{eq:finaluv} hold. Using that $H^1(0,1) \hookrightarrow L^\infty(0,1) \hookrightarrow L^2(0,1)$, the first embedding being compact and the second one continuous, we can apply Aubin-Lions Theorem (see \cite[Section 8, Corollary 4]{Sim87}) to deduce that $W(0,T)\hookrightarrow L^\infty (Q_T)$ compactly. So, at least for subsequence, we have \eqref{eq:convu}, \eqref{eq:convv} and \eqref{eq:convh}. Therefore, the conclusion of \Cref{th:mainresult1} follows from \Cref{prop:ConvSol}. Note that \eqref{eq:finaluw} is guaranteed by \eqref{eq:finaluv} and the continuous embedding $H^1(0,T;L^2(\Omega)) \hookrightarrow C([0,T];L^2(\Omega))$.
\end{proof}

\section{Comments and open problems}

\subsection{Spatial dimension $N>1$ and Dirichlet boundary conditions}

In this paper, we focus on the spatial dimension $N=1$. This is due to the Carleman estimate for heat equation with homogeneous Neumann boundary conditions from \Cref{lem:CarlPDE}. Indeed, this inequality is only proved in $1$-D, see \cite[Appendix A]{KSD18}. This restriction comes from the properties \eqref{eq:deriv_w_0}, \eqref{eq:deriv_w_1} of the weight $\eta$. Constructing such a weight in the multidimensional case is actually an interesting open problem.\\
\indent Note that this type of Carleman estimate from \Cref{lem:CarlPDE} is valid in the multidimensional case for homogeneous Dirichlet boundary conditions, see \cite[Lemma 4.4]{CSRZ14}. Therefore, by a small adaptation of the proof of \Cref{th:mainresult2}, we can also obtain \Cref{th:mainresult2} for homogeneous Dirichlet boundary conditions on the component $v$. Unfortunately, by sending $d_v \rightarrow + \infty$ in the system \eqref{eq:semilinear_diffusion}, $v_{d_v} \rightarrow 0$, so we cannot hope to extend \Cref{th:mainresult1} by this methodology.

\subsection{Other type of nonlinearities}

In this paper, we deal with nonlinear integro-differential equations with cubic nonlinearity, coming from the Gray-Scott model. Actually, we crucially use the fact that the cubic term $uv^2$ is linear in $u$. Indeed, by looking at the definition the space $E$ in \eqref{eq:defE}, we observe that the linear controlled trajectory $(U,V)$, forgetting the weights, satisfy
$$ (U,V) \in H^1(0,T;L^2(\Omega)) \times W(0,T).$$
The extra regularity on the component $V$ is the crucial point to treat the cubic term $UV^2$, see the proof of \Cref{prop:regA}. Moreover, we do not know how to obtain extra regularity on the component $U$ by looking carefully to the proof of \Cref{prop:NullControllabilityLinear}. This is due to the fact that ODEs do not imply automatically spatial regularity, contrary to parabolic PDEs. To sum up, obtaining \Cref{th:mainresult1} and \Cref{th:mainresult2} replacing the Gray-Scott nonlinearity by a nonlinearity superlinear in the variable $u$ is an interesting open problem.

\subsection{Uniform controllability with respect to $d_u\to 0$}

In this paper, we have studied control properties for the ODE-PDE system \eqref{eq:semilinear_diffusion} and obtained a uniform controllability result with respect to the diffusion parameter ${d_v}\to+\infty$, leading us to the ODE-ODE system \eqref{eq:nonlinear_shadow}. An analogous result, where the starting point is \eqref{eq:classical_GSS} (with control on the first equation) and leading us to \eqref{eq:semilinear_diffusion} is in fact an interesting and open question.

On the very recent article \cite{CLW20}, the authors study the limit behavior as $\epsilon \rightarrow 0$ of a chemotaxis problem of the form
\begin{equation}
\begin{cases}
\partial_t u-\Delta u=-\textnormal{div}(u\nabla f(v)) & \text{in } (0,T)\times \Omega, \\
\partial_t v-\epsilon \Delta v=g(u,v) & \text{in } (0,T)\times \Omega,
\end{cases}
\end{equation}
complemented with homogeneous Neumann boundary conditions and where $f\in C^2(\mathbb R)$ and $g\in C^2(\mathbb R^2;\mathbb R)$ are suitable nonlinear functions. Under various assumptions on the initial data and the growth of the nonlinearities, they are able to establish a convergence result  as $\epsilon \to 0$ (see \cite[Theorem 1.4]{CLW20}) towards the system 
\begin{equation}
\begin{cases}
\partial_t u-\Delta u=-\textnormal{div}(u\nabla f(v)) & \text{in } (0,T)\times \Omega, \\
\partial_t \tilde v =g(u,\tilde v) & \text{in } (0,T)\times \Omega.
\end{cases}
\end{equation}

At first glance, it is reasonable to expect that such a result can be adapted to our case to deduce conditions for which the system \eqref{eq:classical_GSS} converges to the uncontrolled version (i.e. $h\equiv0$) of \eqref{eq:semilinear_diffusion}. Nevertheless, designing a sequence of controls $\{h_{\epsilon}\}_{\epsilon}$ for system \eqref{eq:classical_GSS} which is uniformly bounded with respect to $\epsilon$ seems to be a difficult question. 

In fact, from classical papers \cite{CG05,GL07} studying the uniform controllability of vanishing viscosity of parabolic-hyperbolic equation, we can deduce that this is not the case when the support of the control $\omega$ is fixed. This opens up the possibility of obtaining uniform bounds by using moving controls. However, the Carleman estimates developed in \cite{CSRZ14} are not a priori well suited to treat this case. A first approach might be to study a simple 1-D case for a single heat equation using the tools in  \cite{MRR13} and see if a moving control with uniform bounds can be built. The answer is far from obvious.

\appendix

\section{Uniform parabolic Carleman estimate}
\label{sec:uniformcarlHeat}
In this part, we prove \Cref{lem:CarlPDE}. The proof follows the arguments presented in \cite[Appendix A]{KSD18} with some remarks coming grom \cite{chaves_uniform}. For completeness, we give a sketch of the proof. 

We start by giving some useful properties on the weight function $\alpha$ and its derivatives. By simple computations using \eqref{eq:defweights}, we have
\begin{equation} \label{eq:prop_weights}
\begin{split}
\alpha_x&=-\lambda\xi\eta_x, \\ 
\alpha_{xx}&=\lambda^2\xi(-\eta_x^2-\lambda^{-1}\eta_{xx}), \\
|\alpha_{xx}|&\leq C\lambda^2\xi, \\
|\alpha_{xxx}|&\leq C\lambda^3\xi, \\ 
|\alpha_{xxxx}|&\leq C\lambda^4\xi, \\ 
|\alpha_t|&\leq C(T+e^{2\lambda\|\eta\|_\infty})\lambda\xi^2, \\
|\alpha_{xt}|&\leq C(T+1)\lambda^2\xi^2, \\
|\alpha_{tt}|&\leq C(T+T^2+e^{2\lambda\|\eta\|_\infty})\lambda^2\xi^3.
\end{split}
\end{equation}

As usual, we set $w=e^{-s\alpha}\psi$ and compute
\begin{equation*}
w_{x}=-s\alpha_x w+e^{-s\alpha}\psi_x
\end{equation*}
Using the boundary conditions of $\psi$, we deduce $w_x=-s\alpha_xw$ on $(0,T)\times\{0,1\}$. We introduce the parabolic operator $P=\partial_t+\frac{1}{\epsilon}\partial_{xx}$. Then, we have
\begin{equation}\label{eq:iden_carleman}
e^{-s\alpha}P(e^{s\alpha}w)=P_e^{\epsilon}w+P_{k}^{\epsilon}w
\end{equation}
where
\begin{align}\label{eq:p_e}
P_e^{\epsilon}w&=\frac{1}{\epsilon} w_{xx}+s\alpha_t w+\frac{1}{\epsilon}s^2\alpha_x^2 w, \\ \label{eq:p_k}
P_k^{\epsilon}w&=w_t+\frac{2}{\epsilon}s\alpha_xw_x+\frac{1}{\epsilon}s\alpha_{xx}w. 
\end{align}

We take the $L^2$-norm in both sides of \eqref{eq:iden_carleman}, thus obtaining 
\begin{equation}\label{eq:identity_carleman}
\|e^{-s\alpha} P(e^{s\alpha}w)\|_{L^2(Q_T)}^2=\|P_{e}^\epsilon w\|^2_{L^2(Q_T)}+\|P_k^\epsilon w\|^2_{L^2(Q_T)}+2\left(P_e^\epsilon w,P_k^\epsilon w\right)_{L^2(Q_T)}
\end{equation}

A very long, but straightforward computation gives that
\begin{align*}
2\left(P_e^\epsilon w,P_k^\epsilon w\right)_{L^2(Q_T)}= BT+DT_1+DT_2,
\end{align*}
where
\begin{align*}
BT:=& \frac{1}{\epsilon}\left.\int_{0}^{T}s\alpha_{xt}w^2\dt\right|_{0}^{1}+\frac{2}{\epsilon^2}\left.\int_{0}^{T}s\alpha_{x}w_x^2\dt\right|_{0}^{1}+\frac{2}{\epsilon^2}\left.\int_{0}^{T}s\alpha_{xx}w_x w\dt\right|_{0}^{1} \\
&-\frac{1}{\epsilon^2}\left.\int_{0}^{T}s\alpha_{xxx}w^2\dt\right|_{0}^{1}+\frac{2}{\epsilon}\left.\int_{0}^{T}s^2\alpha_t\alpha_{x}w^2\dt\right|_{0}^{1}+\frac{2}{\epsilon^2}\left.\int_{0}^{T}s^3\alpha_{x}^3 w^2\dt\right|_{0}^{1}, \\
DT_1:=&-\frac{4}{\epsilon^2}\iint_{Q_T} s\alpha_{xx}w_x^2\dx\dt, \\
DT_2:=& -\frac{4}{\epsilon}\iint_{Q_T}s^2\alpha_{tx}\alpha_x w^2\dx\dt+\frac{1}{\epsilon^2}\iint_{Q_T}s\alpha_{xxxx}w^2\dx\dt-\iint_{Q_T}s\alpha_{tt}w^2\dx\dt\\
&-\frac{4}{\epsilon^2}\iint_{Q_T} s^3\alpha_x^2\alpha_{xx}w^2\dx\dt.
\end{align*}

Using that $w_{x}=-s\alpha_x w$ in $(0,T)\times\{0,1\}$, we can obtain
\begin{align*}
BT=&\frac{4}{\epsilon^2}\left.\int_{0}^T s^3\alpha_x^3 w^2 \dt\right|_{0}^{1}+\frac{2}{\epsilon}\left.\int_{0}^{T}s^2\alpha_t\alpha_{x}w^2\dt\right|_{0}^{1}-\frac{2}{\epsilon^2}\left.\int_{0}^{T}s^2\alpha_x\alpha_{xx}w^2\dt\right|_{0}^{1} \\
&+\frac{1}{\epsilon}\left.\int_{0}^{T}s\alpha_{xt}w^2\dt\right|_{0}^{1}-\frac{1}{\epsilon^2}\left.\int_{0}^{T}s\alpha_{xxx}w^2\dt\right|_{0}^{1}.
\end{align*}
Moreover, using properties \eqref{eq:prop_weights} together with \eqref{eq:deriv_w_0}--\eqref{eq:deriv_w_1}, it is not difficult to see that
\begin{align*}
BT\geq& \frac{4C}{\epsilon^2}\int_0^Ts^3\lambda^3\xi^3w^2\dt\Big|_{x=1}+\frac{4C}{\epsilon^2}\int_0^Ts^3\lambda^3\xi^3w^2\dt\Big|_{x=0} \\
&-\frac{2C}{\epsilon^2}\int_{0}^{T}s^2\lambda^2\xi^3(T+e^{2\lambda\|\eta\|_\infty})w^2\dt\Big|_{x=1}-\frac{2C}{\epsilon^2}\int_{0}^{T}s^2\lambda^2\xi^3(T+e^{2\lambda\|\eta\|_\infty})w^2\dt\Big|_{x=0} \\
&-\frac{2}{\epsilon^2}\int_{0}^{T}s^2\lambda^3\xi^2w^2\dt\Big|_{x=1}-\frac{2}{\epsilon^2}\int_{0}^{T}s^2\lambda^3\xi^2w^2\dt\Big|_{x=0} \\
&-\frac{C}{\epsilon^2}\int_{0}^{T}s(T+1)\lambda^2\xi^2w^2\dt\Big|_{x=1}-\frac{C}{\epsilon^2}\int_{0}^{T}s(T+1)\lambda^2\xi^2w^2\dt\Big|_{x=0} \\
&-\frac{C}{\epsilon^2}\int_{0}^{T}s\lambda^3\xi w^2\dt\Big|_{x=1}-\frac{C}{\epsilon^2}\int_{0}^{T}s\lambda^3\xi w^2\dt\Big|_{x=0}
\end{align*}
where we have used that $0<\epsilon\leq 1$ to adjust the powers of $\epsilon$. Finally, taking $\lambda\geq C$ and $s\geq C(1+T+e^{2\lambda\|\eta\|_\infty})$, we get
\begin{equation}\label{eq:estimate_BT}
BT\geq \frac{C}{\epsilon^2}\int_0^Ts^3\lambda^3\xi^3w^2\dt\Big|_{x=1}+\frac{C}{\epsilon^2}\int_0^Ts^3\lambda^3\xi^3w^2\dt\Big|_{x=0}.
\end{equation}

Let us focus now on the distributed terms. Using \eqref{eq:prop_weights}, we can rewrite
\begin{equation*}
DT_1=\frac{4}{\epsilon^2}\iint_{Q_T}s\lambda^2\xi\eta_x^2w_x^2\dx\dt+\frac{4}{\epsilon^2}\iint_{Q_T}s\lambda\xi\eta_{xx}w_x^2\dx\dt
\end{equation*}
and
\begin{align*}
DT_2=&-\frac{4}{\epsilon}\iint_{Q_T}s^2\alpha_{xt}\alpha_x w^2\dx\dt+\frac{1}{\epsilon^2}\iint_{Q_T}s\alpha_{xxxx} w^2\dx\dt \\
&-\iint_{Q_T}s\alpha_{tt}w^2\dx\dt+\frac{4}{\epsilon^2}\iint_{Q_T}s^3\lambda^3\xi^3\eta_x^2\eta_{xx}w^2\dx\dt+\frac{4}{\epsilon^2}\iint_{Q_T}s^3\lambda^4\xi^3\eta_x^4w^2\dx\dt.
\end{align*}
Using estimates \eqref{eq:prop_weights}, we can bound by below as follows
\begin{equation}\label{eq:est_DT1}
DT_1\geq \frac{4}{\epsilon^2}\iint_{Q_T}s\lambda^2\xi\eta_x^2w_x^2\dx\dt-\frac{C}{\epsilon^2}\iint_{Q_T}s\lambda\xi w_x^2\dx\dt.
\end{equation}
and
\begin{align}\notag
DT_2\geq & \frac{4}{\epsilon^2}\iint_{Q_T}s^3\lambda^4\xi^3\eta_x^4w^2\dx\dt-\frac{C(T+1)}{\epsilon}\iint_{Q_T}s^2\lambda^3\xi^3 w^2\dx\dt -\frac{C}{\epsilon^2}\iint_{Q_T}s\lambda^4\xi w^2\dx\dt \\ \label{eq:est_DT2}
&-C\left(T+T^2+e^{2\lambda\|\eta\|_\infty}\right)\iint_{Q_T}s\lambda^2\xi^3w^2\dx\dt-\frac{C}{\epsilon^2}\iint_{Q_T}s^3\lambda^3\xi^3w^2\dx\dt.
\end{align}
Collecting estimates \eqref{eq:estimate_BT}--\eqref{eq:est_DT2} yield
\begin{align*}
2\left(P_e^\epsilon w,P_k^\epsilon w\right)_{L^2(Q_T)} \geq &  \frac{C}{\epsilon^2}\int_0^Ts^3\lambda^3\xi^3w^2\dt\Big|_{x=1}+\frac{C}{\epsilon^2}\int_0^Ts^3\lambda^3\xi^3w^2\dt\Big|_{x=0} \\
&+\frac{4}{\epsilon^2}\iint_{Q_T}s\lambda^2\xi\eta_x^2w_x^2\dx\dt+\frac{4}{\epsilon^2}\iint_{Q_T}s^3\lambda^4\xi^3\eta_x^4w^2\dx\dt-\frac{C}{\epsilon^2}\iint_{Q_T}s\lambda\xi w_x^2\dx\dt \\
&-\frac{C(T+1)}{\epsilon}\iint_{Q_T}s^2\lambda^3\xi^3 w^2\dx\dt -\frac{C}{\epsilon^2}\iint_{Q_T}s\lambda^4\xi w^2\dx\dt \\
&-C\left(T+T^2+e^{2\lambda\|\eta\|_\infty}\right)\iint_{Q_T}s\lambda^2\xi^3w^2\dx\dt-\frac{C}{\epsilon^2}\iint_{Q_T}s^3\lambda^3\xi^3w^2\dx\dt.
\end{align*}
Using property \eqref{eq:deriv_eta} and taking $\lambda\geq C$ and $s\geq C(T+T^2)$, we deduce
\begin{equation}\label{eq:est_prod}
\begin{split}
2\left(P_e^\epsilon w,P_k^\epsilon w\right)_{L^2(Q_T)}&+\frac{1}{\epsilon^2}\iint_{\omega_{0}(t)\times(0,T)}s^3\lambda^4\xi^3 w^2 \dx\dt+\frac{1}{\epsilon^2}\iint_{\omega_{0}(t)\times(0,T)}s\lambda^2\xi w_x^2 \dx\dt \\
\geq &  \frac{C}{\epsilon^2}\int_0^Ts^3\lambda^3\xi^3w^2\dt\Big|_{x=1}+\frac{C}{\epsilon^2}\int_0^Ts^3\lambda^3\xi^3w^2\dt\Big|_{x=0} \\
&+\frac{C}{\epsilon^2}\iint_{Q_T}s\lambda^2\xi w_x^2\dx\dt+\frac{C}{\epsilon^2}\iint_{Q_T}s^3\lambda^4\xi^3 w^2\dx\dt.
\end{split}
\end{equation}

Combining \eqref{eq:est_prod} with \eqref{eq:identity_carleman}, we get
\begin{equation}\label{eq:car_ineq_0}
\begin{split}
C\|e^{-s\alpha} &P(e^{s\alpha}w)\|_{L^2(Q_T)}^2+\frac{C}{\epsilon^2}\iint_{\omega_{0}(t)\times(0,T)}s^3\lambda^4\xi^3 w^2 \dx\dt+\frac{C}{\epsilon^2}\iint_{\omega_{0}(t)\times(0,T)}s\lambda^2\xi w_x^2 \dx\dt \\
\geq& \ \|P_{e}^\epsilon w\|^2_{L^2(Q_T)}+\|P_k^\epsilon w\|^2_{L^2(Q_T)}+\frac{1}{\epsilon^2}\int_0^Ts^3\lambda^3\xi^3w^2\dt\Big|_{x=1}+\frac{1}{\epsilon^2}\int_0^Ts^3\lambda^3\xi^3w^2\dt\Big|_{x=0} \\
&+\frac{1}{\epsilon^2}\iint_{Q_T}s\lambda^2\xi w_x^2\dx\dt+\frac{1}{\epsilon^2}\iint_{Q_T}s^3\lambda^4\xi^3 w^2\dx\dt.
\end{split}
\end{equation}

We will add terms corresponding to $w_{xx}$ and $w_t$ to the right-hand side of \eqref{eq:car_ineq_0}. For this, we multiply \eqref{eq:p_e} by $s^{-1/2}\xi^{-1/2}$ and take the $L^2$-norm, that is,
\begin{align}\notag 
\frac{1}{\epsilon^2}s^{-1}\|\xi^{-1/2}w_{xx}\|^2_{L^2(Q_T)}&=s^{-1}\|\xi^{-1/2}(P_e^{\epsilon}w-s\alpha_tw-\frac{1}{\epsilon}s^2\alpha_x^2 w)\|^2_{L^2(Q_T)} \\ \notag 
&\leq C s^{-1}\iint_{Q_T}\xi^{-1}|P_e^{\epsilon}w|^2\dx\dt+CT^2\iint_{Q_T}s^2\lambda^2\xi^3w^2\dx\dt\\ \notag
&\quad +\frac{C}{\epsilon^2}s^3\iint_{Q_T}\lambda^4\xi^3 w^2\dx\dt \\ \label{eq:est_wxx}
& \leq C s^{-1}\iint_{Q_T}\xi^{-1}|P_e^{\epsilon}w|^2\dx\dt+\frac{C}{\epsilon^2}\iint_{Q_T}s^3\lambda^4\xi^3w^2\dx\dt
\end{align}
where we have used that $\epsilon\leq 1$ and $s\geq CT^2$ in the last line. Arguing in the same way and considering \eqref{eq:p_k}, it is not difficult to see that
\begin{align}\notag
s^{-1}\|\xi^{-1/2}w_t\|^2_{L^2(Q_T)}\leq&\  Cs^{-1}\iint_{Q_T}\xi^{-1}|P_k^{\epsilon}w|^2\dx\dt+\frac{C}{\epsilon^2}\iint_{Q_T}s\lambda^2\xi w_{x}^2\dx\dt \\\label{eq:est_wt}
& +\frac{C}{\epsilon^2}\iint_{Q_T}s\lambda^4\xi w^2\dx\dt.
\end{align}
Using that $s^{-1}\xi^{-1}\leq C$ for all $(t,x)\in Q_T$, we can use \eqref{eq:est_wxx} and \eqref{eq:est_wt} to estimate from below in \eqref{eq:car_ineq_0} and obtain
\begin{align}\notag 
C&\|e^{-s\alpha} P(e^{s\alpha}w)\|_{L^2(Q_T)}^2+\frac{C}{\epsilon^2}\iint_{\omega_{0}(t)\times(0,T)}s^3\lambda^4\xi^3 w^2 \dx\dt+\frac{C}{\epsilon^2}\iint_{\omega_{0}(t)\times(0,T)}s\lambda^2\xi w_x^2 \dx\dt \\ \notag
\geq& \ \frac{1}{\epsilon^2}\int_0^Ts^3\lambda^3\xi^3w^2\dt\Big|_{x=1}+\frac{1}{\epsilon^2}\int_0^Ts^3\lambda^3\xi^3w^2\dt\Big|_{x=0} +\frac{1}{\epsilon^2}\iint_{Q_T}s\lambda^2\xi w_x^2\dx\dt\\ \label{eq:des_car_fm}
& +\frac{1}{\epsilon^2}\iint_{Q_T}s^3\lambda^4\xi^3w^2\dx\dt+ \frac{1}{\epsilon^2}\iint_{Q_T}s^{-1}\xi^{-1}|w_{xx}|^2\dx\dt+\iint_{Q_T}s^{-1}\xi^{-1}|w_{t}|^2\dx\dt.
\end{align}

To conclude the proof, we need to eliminate the local term of $w_x$ in the above equation. For this, consider a function $\zeta\in C^\infty([0,T]\times[0,L])$ verifying 
\begin{equation*}
\begin{cases}
0\leq \zeta \leq 1 &\forall (t,x)\in [0,T]\times[0,L], \\
\zeta(t,x)=1 &\forall t\in [0,T],  \;\; \forall x\in \omega_0(t),\\
\zeta(t,x)=0 &\forall t\in[0,T], \;\; \forall x\in[0,L]\setminus\overline{\omega_1(t)}.
\end{cases}
\end{equation*}
We have
\begin{align*}
\frac{1}{\epsilon^2}\iint_{\omega_0(t)\times(0,T)}s\lambda^2\xi w_{x}^2\dx\dt & \leq \frac{1}{\epsilon^2}\iint_{Q_T}\zeta s\lambda^2\xi w_{x}^2\dx\dt \\
&=-\frac{1}{\epsilon^2}\int_{0}^{T}s\lambda^2\xi\zeta w w_{x}\dt\Big|_{0}^{1}-\frac{1}{\epsilon^2}\iint_{Q_T}s\lambda^2\xi ww_{xx}\zeta\dx\dt \\
&\quad -\frac{1}{\epsilon^2}\iint_{Q_T}s\lambda^2\xi_{x} ww_{x}\zeta\dx\dt-\frac{1}{\epsilon^2}\iint_{Q_T}s\lambda^2\xi ww_{x}\zeta_{x}\dx\dt.
\end{align*}
Using the fact that $w_x=-s\alpha_x w$ on $(0,T)\times\{0,1\}$ together with properties \eqref{eq:deriv_w_0}--\eqref{eq:deriv_w_1} and Cauchy-Schwarz and Young inequalities, we get
\begin{align*}
\frac{C}{\epsilon^2}\iint_{w_0(t)\times(0,T)}s\lambda^2\xi w_x^2\leq& \ \frac{1}{2\epsilon^2}\iint_{Q_T}s^{-1}\xi^{-1}|w_{xx}|^2\dx\dt+\frac{1}{2\epsilon^2}\iint_{Q_T}s\lambda^2\xi w_x^2\dx\dt \\
&+\frac{C}{\epsilon^2}\iint_{\omega_1(t)\times(0,T)}s^3\lambda^4\xi^3w^2\dx\dt.
\end{align*}
Using the above estimate in \eqref{eq:des_car_fm} and recalling the change of variables $w=e^{-s\alpha}\psi$ gives the desired result. This concludes the proof.

\section{Proof of a precise observability inequality from the Carleman estimate}\label{app:obs}
The goal of this part is to prove \Cref{prop:UniformCarlemanTer}.
\begin{proof}
The proof is by now standard and relies on well-known arguments, we follow here the presentation in \cite[Lemma 4.1]{Gue13}. We fix $\lambda, s$ large enough such that the Carleman estimate from \Cref{prop:UniformCarleman} holds and such that we have the following estimate
\begin{equation}
\label{eq:compareweightslambda}
e^{-4s\alpha^\star+2s\widehat{\alpha}}\leq C e^{-s\alpha}\leq C e^{-s\alpha^\star}
\end{equation}
All the following positive constants $C>0$ can now depend on $\lambda,s$.\\
\indent By construction, $\alpha=\beta$ and $\xi=\gamma$ in $(T/2,T)\times (0,1)$, therefore
\begin{align}\notag 
&\int_{T/2}^{T} \int_0^1  \gamma |\phi|^2 e^{-2s \beta} \dx \dt +\int_{T/2}^{T} \int_0^1  \gamma^3 |\psi|^2 e^{-2s \beta} \dx \dt \\ \label{eq:iden_weights}
=& \int_{T/2}^{T} \int_0^1  \xi |\phi|^2 e^{-2s \alpha} \dx \dt+\int_{T/2}^{T} \int_0^1  \xi^3 |\psi|^2 e^{-2s \alpha} \dx \dt
\end{align}
Moreover, from the definition of $\beta^\star$ and $\alpha^\star$, we readily see that $e^{-s\alpha^\star}\leq e^{-s\beta^\star}$. 

From this fact and noting that $\beta$ (resp. $\alpha$) blows up exponentially as $t\to T^{-}$ (resp. $t\to 0^+$ and $t\to T^{-}$) while $\gamma$ (resp. $\xi$) blows up polynomially as $t\to T^{-}$ (resp. $t\to 0^+$ and $t\to T^{-}$), we can use \eqref{eq:iden_weights}, \eqref{eq:compareweightslambda} and our Carleman inequality \eqref{eq:car_Prop} to deduce that
\begin{align}\notag 
\int_{T/2}^{T} \int_0^1&  \gamma |\phi|^2 e^{-2s \beta} \dx \dt  + \int_{T/2}^{T} \int_0^1  \gamma^3 |\psi|^2 e^{-2s \beta} \dx \dt \\ \notag
\leq & \ C(s,\lambda) \Bigg( \iint_{Q_T}e^{-2s\beta}\gamma^3|g_1|^2\dx\dt + \iint_{Q_T}e^{-2s\beta} |g_2|^2\dx\dt \\ \label{eq:estimate_T2_T} 
&\qquad\qquad + \iint_{\omega_2\times(0,T)}  e^{-s\beta^\star}({\gamma}^\star)^{8}|\phi|^2\dx\dt \Bigg) .
\end{align}

For the set $(0,T/2)\times(0,1)$, we will use energy estimates for the system \eqref{eq:linear_adj}. More precisely, consider the function $\nu\in C^1([0,T])$ such that 
\begin{equation}
\nu=1 \text{ in } [0,T/2], \quad \nu=0 \text{ in }[3T/4,T/4], \quad |\nu^\prime(t)|\leq C/T.
\end{equation}
Setting  $(\widetilde \phi,\widetilde \psi)=(\nu\phi,\nu\psi)$, it is not difficult to see that the new variables verify the system
\begin{equation}\label{eq:adj_cut}
\begin{cases}
-\widetilde{\phi}_t=a_{11} \widetilde \phi + a_{21}\widetilde \psi + \nu g_1-\nu^\prime\phi & \text{in } Q_T, \\
-\widetilde \psi_t={d_v}\widetilde \psi_{xx} +a_{12}\widetilde \phi+a_{22} \widetilde \psi +\nu g_2 - \nu^\prime \psi &\text{in } Q_T, \\
\partial_x \widetilde \psi=0 &\text{on } \Sigma_T, \\
(\widetilde \phi,\widetilde \psi)(T,\cdot)=(0,0) &\text{in }(0,1).
\end{cases}
\end{equation}
From standard energy estimates, we deduce that system \eqref{eq:adj_cut} verifies
\begin{align} \notag
\int_{\Omega}&|\widetilde \phi(t)|^2\dx+\int_{\Omega}|\widetilde \psi(t)|^2\dx+{d_v} \int_{t}^{T}\!\!\!\int_{\Omega}|\widetilde \psi_x|^2\dx\dt \\ \notag 
\leq & \ C\left(\int_{t}^{T}\!\!\!\int_{\Omega}|\widetilde \phi|^2\dx\dt+\int_{t}^{T}\!\!\!\int_{\Omega}|\widetilde \psi|^2\dx\dt+\int_{t}^{T}\!\!\!\int_{\Omega}|\eta g_1|^2\dx\dt+\int_{t}^{T}\!\!\!\int_{\Omega}|\eta g_2|^2\dx\dt\right. \\ 
&\quad \left. + \int_{t}^{T}\!\!\!\int_{\Omega}|\eta^\prime \phi|^2\dx\dt +\int_{t}^{T}\!\!\!\int_{\Omega}|\eta^\prime \psi|^2\dx\dt  \right), \quad \forall t\in [0,T],\notag
\end{align}
where $C$ is a positive constant only depending on $a_{ij}$. 

Dropping the third term in the left-hand side of the above expression, we use Gronwall's inequality to deduce
\begin{align*}
&\|\widetilde \phi(0)\|^2_{L^2(\Omega)}  + \|\widetilde \psi(0)\|^2_{L^2(\Omega)}  + \iint_{Q_T}|\widetilde \phi|^2 \dx\dt + \iint_{Q_T}|\widetilde \psi|^2\dx\dt \\
& \leq  C\left(\iint_{Q_T} |\eta g_1|^2\dx\dt+\iint_{Q_T}|\eta g_2|^2\dx\dt + \iint_{Q_T}|\eta^\prime \phi|^2 \dx\dt+\iint_{Q_T}|\eta^\prime \psi|^2\dx\dt \right)
\end{align*}
for some constant $C>0$ only depending on $T$ and $a_{ij}$.

Recalling the definition of $\eta$, we obtain from the above expression 
\begin{align*}
\| \phi(0)\|^2_{L^2(\Omega)} & + \| \psi(0)\|^2_{L^2(\Omega)}  + \int_{0}^{T/2}\!\!\!\!\int_{\Omega} | \phi|^2  \dx\dt + \int_{0}^{T/2}\!\!\!\!\int_{\Omega}| \psi|^2\dx\dt \\
& \leq  C\left(\int_{0}^{3T/4}\!\!\!\!\int_{\Omega}| g_1|^2 \dx\dt+\int_{0}^{3T/4}\!\!\!\!\int_{\Omega}| g_2|^2\dx\dt\right) \\
&\quad  + \frac{C}{T^2}\left( \int_{T/2}^{3T/4}\!\!\!\!\int_{\Omega} |\phi|^2 \dx\dt+\int_{T/2}^{3T/4}\!\!\!\!\int_{\Omega}| \psi|^2\dx\dt \right).
\end{align*}
Since the domain of integration in the above integrals is away from the singularity of the weight functions \eqref{eq:weight_cut} at $t=T$ (and therefore they are bounded), we can introduce them in the above inequality as follows
\begin{align*}
&\| \phi(0)\|^2_{L^2(\Omega)}  + \| \psi(0)\|^2_{L^2(\Omega)}  \\
&\quad + \int_{0}^{T/2}\!\!\!\!\int_{\Omega} e^{-2s\beta}  \gamma \left(| \phi|^2 + | \phi_x|^2 \right)\dx\dt + \int_{0}^{T/2}\!\!\!\!\int_{\Omega} \gamma^3 | \psi|^2 e^{-2s\beta}\dx\dt \\
& \leq  C(s,\lambda,T)\left(\int_{0}^{3T/4}\!\!\!\!\int_{\Omega} e^{-2s\beta}   \gamma^3 | g_1|^2 \dx\dt+\int_{0}^{3T/4}\!\!\!\!\int_{\Omega}e^{-2s\beta} | g_2|^2\dx\dt\right) \\
&\quad  + C(s,\lambda,T)\left( \int_{T/2}^{3T/4}\!\!\!\!\int_{\Omega} e^{-2s\beta}\gamma |\phi|^2 \dx\dt +\int_{T/2}^{3T/4}\!\!\!\!\int_{\Omega} e^{-2s\beta}\gamma^3| \psi_2|^2\dx\dt \right).
\end{align*}

Using estimate \eqref{eq:estimate_T2_T} to bound all the terms on the last line of the above inequality and adding up the resulting expression to \eqref{eq:estimate_T2_T} yields
\begin{align*}
&\| \phi(0)\|^2_{L^2(\Omega)} +\| \psi(0)\|^2_{L^2(\Omega)} + \iint_{Q_T} e^{-2s\beta}  \gamma | \phi|^2 \dx\dt + \iint_{Q_T} \gamma^3 | \psi|^2 e^{-2s\beta}\dx\dt \\
& \leq  C\left(\iint_{Q_T} e^{-2s\beta}   \gamma^3 | g_1|^2 \dx\dt+\iint_{Q_T} e^{-2s\beta} | g_2|^2\dx\dt\right) \\
& \quad +C\left(\iint_{\omega_2(t)\times(0,T)}  e^{-s\beta^\star}({\gamma}^\star)^{8}|\phi|^2 \dx\dt\right).
\end{align*}
To conclude, it is enough to use definitions \eqref{eq:weight_cut} in the above inequality. This ends the proof. 
\end{proof}

\section{Some properties of the heat semigroup}
We recall in the next result some well-known facts about the heat semigroup with a diffusion parameter $d_v >0 $ with homogeneous Neumann boundary conditions on the interval $(0,1)$, denoted by $\{e^{t{d_v}\partial_{xx}}\}_{t\geq 0}$. The proof can be found for instance in \cite[Lemma A.1]{MCHKS18}.
\begin{lem}\label{lem:semi_properties}The following properties hold true.
\begin{enumerate}
\item[a.] For every constant $K\in\mathbb R$, we have $e^{t{d_v} \partial_{xx}}K=K$ for all $t\geq 0$.
\item[b.] For every $z_0\in L^2(\Omega)$, there exits a constant $C>0$ only depending on $z_0$ such that for every $t \geq 0$,
\begin{equation}
\label{eq:estisemigroupNeumann}
\norme{e^{t{d_v}\partial_{xx}}\left(z_0-\int_{\Omega}z_0\dx\right)}_{L^2(\Omega)}\leq C e^{-\lambda_1 d_v t} \norme{z_0}_{L^2(\Omega)},
\end{equation}
where $\lambda_1 >0$ is the first positive eigenvalue of the Neumann Laplacian operator $-\partial_{xx}$ on $(0,1)$.
\end{enumerate}
\end{lem}

%
%

\bibliographystyle{alpha}
\small{\bibliography{bib_ODE_ODE}}

\end{document}